\documentclass[pdflatex,sn-mathphys-num]{sn-jnl}
\theoremstyle{thmstyleone}%
\newtheorem{theorem}{Theorem}
\newtheorem{proposition}[theorem]{Proposition}%

\theoremstyle{thmstyletwo}%
\newtheorem{remark}{Remark}%

\theoremstyle{thmstylethree}%
\newtheorem{definition}{Definition}%

\usepackage{amssymb, amsmath, amsthm, mathtools}
\usepackage{mathrsfs}
\usepackage{hyperref,bookmark}
\usepackage{graphicx}
\usepackage{epsfig,color}
\usepackage{float}
\usepackage{caption}
\usepackage{subcaption}
\usepackage{verbatim}
\usepackage{dcolumn}
\usepackage{tikz}
\usepackage{url}
\usepackage[utf8x]{inputenc}
\usepackage[english]{babel}
\usepackage{epsfig}
\usepackage{listings}
\usepackage{paralist}
\usepackage{bm}
\usepackage{oldgerm}
\usepackage{geometry}
\usepackage{relsize}
\usepackage{afterpage}
\usepackage{lmodern}
\usepackage{tabularx}
\usepackage{cancel}
\usepackage{natbib} 
\usepackage[overload]{empheq}

\newcommand{\dd}{\mathrm{d}}
\newcommand{\ex}{\mathrm{e}}

\newcommand{\vt}{\vartheta}

\newcommand{\trans}{^\mathsf{T}}
\newcommand{\transl}{\mathscr{V}}
\newcommand{\av}{\bm{a}}

\newcommand{\cv}{\bm{c}}

\newcommand{\dv}{\bm{d}}
\newcommand{\e}{\bm{e}}
\newcommand{\h}{\bm{h}}
\newcommand{\x}{\bm{x}}
\newcommand{\zero}{\bm{0}}
\newcommand{\y}{\bm{y}}

\newcommand{\normal}{\bm{\nu}}
\newcommand{\surface}{\mathscr{S}}
\newcommand{\tplane}{\mathscr{T}}

\newcommand{\nablas}{\nabla\!_\mathrm{s}}

\newcommand{\nablast}{\nabla^\ast\!\!\!_\mathrm{s}}

\newcommand{\laps}{\mathop{{}\bigtriangleup_\mathrm{s}\!}\nolimits}
\newcommand{\nablastwo}{\nabla^2\!\!\!\!_\mathrm{s}\,}

\newcommand{\uv}{\bm{u}}

\newcommand{\f}{\bm{f}}

\renewcommand{\f}{\bm{f}}
\newcommand{\rv}{\bm{r}}

\newcommand{\bv}{\bm{b}}

\newcommand{\vv}{\bm{v}}

\newcommand{\W}{\mathbf{W}}
\newcommand{\F}{\mathbf{F}}
\newcommand{\G}{\mathbf{G}}
\renewcommand{\H}{\mathbf{H}}
\newcommand{\I}{\mathbf{I}}

\newcommand{\R}{\mathbf{R}}
\newcommand{\Rd}{\R_\mathrm{d}}
\newcommand{\Rb}{\R_\mathrm{b}}
\newcommand{\Q}{\mathbf{Q}}
\newcommand{\C}{\mathbf{C}}
\newcommand{\B}{\mathbf{B}}
\newcommand{\U}{\mathbf{U}}
\newcommand{\V}{\mathbf{V}}

\newcommand{\curl}{\operatorname{curl}}
\newcommand{\curls}{\curl_\mathrm{s}}

\newcommand{\divs}{\diver_\mathrm{s}\!}

\newcommand{\tr}{\operatorname{tr}}

\newcommand{\skw}{\operatorname{skw}}

\newcommand{\curvature}{(\nablas\normal)}
\newcommand{\A}{\mathbf{A}}

\newcommand{\euclid}{\mathscr{E}}
\newcommand{\framec}{(\e_1,\e_2,\e_3)}
\newcommand{\framecyl}{(\e_r,\e_\vt,\e_z)}
\newcommand{\frameen}{(\e_1,\e_2,\normal)}

\newcommand{\frameprime}{(\e_1^\prime,\e_2^\prime,\normal)}

\newcommand{\framee}{(\e_u,\e_v,\normal)}
\newcommand{\frameu}{(\uv_1,\uv_2,\normal)}
\newcommand{\framev}{(\vv_1,\vv_2,\normal^\ast)}

\newcommand{\curve}{\bm{x}}

\newcommand{\orth}{\mathsf{SO}(3)}

\newcommand{\proj}{\mathbf{P}(\normal)}

\newcommand{\sphere}{\mathbb{S}^2}

\newcommand{\diver}{\operatorname{div}} 

\newcommand{\rvu}{\bm{r}_u}
\newcommand{\rvv}{\bm{r}_v}

\newcommand{\wdr}{w_\mathrm{d}}
\newcommand{\wbe}{w_\mathrm{b}}
\newcommand{\wst}{w_\mathrm{s}}

\begin{document}

\title[]{Metric Implications in the Kinematics of Surfaces}


\author[1]{\fnm{Andr\'e M.} \sur{Sonnet}}\email{andre.sonnet@strath.ac.uk}

\author*[2]{\fnm{Epifanio G.} \sur{Virga}}\email{eg.virga@unipv.it}
\equalcont{These authors contributed equally to this work.}

\affil*[1]{\orgdiv{Department of Mathematics and Statistics}, \orgname{University of Strathclyde}, \orgaddress{\street{26 Richmond Street}, \city{Glasgow}, \postcode{G1 1XH}, \country{U.K}}}

\affil*[2]{\orgdiv{Dipartimento di Matematica}, \orgname{Universit\`a di Pavia}, \orgaddress{\street{Via Ferrata 5}, \city{Pavia}, \postcode{27100}, \country{Italy}}}

\abstract{In the direct approach to continua in reduced space dimensions, a thin shell is described as a mathematical surface in three-dimensional space. An exploratory kinematic study of such surfaces could be very valuable, especially if conducted with no use of coordinates. Three energy contents have been identified in a thin shell, which refer to three independent deformation modes: stretching, drilling, and bending. We analyze the consequences for the three energy contents produced by metric restrictions imposed on the admissible deformations. Would the latter stem from physical constraints, the elastic response of a shell could be hindered in ways that might not be readily expected.}

\keywords{Kinematics of surfaces, Soft shells, Thin shells, Deformation modes, Soft elasticity, Eversions.}


\pacs[MSC Classification]{74A05 74B20 74K25}

\maketitle

\section{Introduction}\label{sec:intro}
It has appropriately been remarked that the kinematics of material surfaces cannot be reduced to the standard treatment of surfaces made in classical differential geometry (see, for example, \cite{seguin:coordinate}). As pointed out in \cite{murdoch:direct} (see, in particular, the excerpt reproduced in \cite{seguin:coordinate}, and also \cite{murdoch:coordinate}), this does not merely amount to renouncing the use of coordinates (which are at the heart of most analyses, as shown in \cite{ciarlet:introduction}). It is rather a matter of prioritising physical essence over mathematical commodity.

This paper follows this line of thought. Motivated by a direct theory for soft thin shells proposed in \cite{sonnet:variational}, we explore the consequences of certain kinematic restrictions, mainly of a metric nature, on the energy contents, which we classify in three independent modes: stretching, drilling, and bending. Our development is intrinsic (coordinate-free) and generically tensorial; it employs the method of moving frames, phrased in the language of \emph{connectors}, surface tangential vector fields that play here the role played by differential forms in Cartan's method \cite{cartan:methode}.\footnote{\label{ftn:history} The method of moving frames has an interesting history, which predates Cartan's work. It can be traced back to Darboux (see pp.\,47-57 of \cite{darboux:lecons_I} and also \cite{ribaucour:memoire}), who extended to surfaces the classical Frenet-Serret formulas for curves. A first generalization of Darboux's method is due to Cotton~\cite{cotton:generalisation}; apparently, \cite{cartan:structure} is the first paper by Cartan on this topic (see also \cite{chern:moving}).}

In an attempt to make this paper self-contained and attract readers not yet conversant with the method of moving frames, we recall some general preliminaries in Sect.~\ref{sec:prelim} and give a special account on connectors in Sect.~\ref{sec:connectors}. Section~\ref{sec:kinematics} is devoted to the general kinematics of material surfaces, with special emphasis on an \emph{invariant} rotation gradient, a third-rank tensor with a special status in the energetics of softy shells. In Sects.~\ref{sec:conformal} and \ref{sec:isometric}, we consider two classes of metric restrictions on the deformation of a surface, presented in order of increasing severity. In particular, in Sect.~\ref{sec:isometric}, we analyze the implications for the independent energy contents of the requirement that the deformation be isometric. Finally, in Sect.~\ref{sec:conclusions}, we collect the main conclusions of our study. The paper is closed by four appendices with ancillary results and computational details.

\section{Preliminaries on surface calculus}\label{sec:prelim}
With the primary interest to make our development self-contained, we collect here a few mathematical properties of surfaces, phrased mainly in the language adopted in the method of \emph{moving frames}, although our method differs formally from the latter in that it uses vector fields instead of differential forms.\footnote{Our method is inspired by the work of Weatherburn \cite{weatherburn:differential_1,weatherburn:differential_2}, whose essential features are also succinctly outlined in \cite{sonnet:bending-neutral}. His work was indeed preceded by the introduction of a general vector method in differential geometry by the Italian school that originated from Levi-Civita (see \cite{burali-forti:fondamenti,burgatti:teoremi,burgatti:memorie,burgatti:analisi} for the relevant contributions, which  Weatherburn seems to have been unaware of.)}

Our approach to surface calculus will be \emph{absolute}, that is, it will avoid, insofar as possible, explicit resort to coordinates, patches, and atlases of local maps.

Let $\surface$ be a (locally) smooth (say, of class $C^3$), orientable surface imbedded in three-dimensional space $\euclid$.\footnote{In local coordinates $(u,v)$ ranging in a domain $\Omega\subset\mathbb{R}^2$, $\surface$ can be split into \emph{patches}, defined as applications of $\Omega$ into $\euclid$ that have the requested degree of regularity, are one-to-one, and have continuous inverse \cite[see, for example,][p.\,130]{o'neill:elementary}.} We denote by $\normal$ the unit normal on $\surface$ in our chosen orientation.

A major role is played below by the notion of \emph{surface gradient} $\nablas$, which we introduce with the aid of smooth curves $\curve(t)$ on $\surface$. A scalar field $\varphi:\surface\to\mathbb{R}$ is differentiable on $\surface$ whenever we can write
\begin{equation}
	\label{eq:differentiability_scalar}
	\frac{\dd}{\dd t}\varphi(\curve(t))=\nablas\varphi\cdot\dot{\curve},
\end{equation}
where
the vector $\nablas\varphi$, which is perpendicular to the normal $\normal$, is the surface gradient of $\varphi$,
a superimposed dot $\dot{\null}$ denotes differentiation with respect to the parameter $t$, and $\dot{\curve}$ is a  vector along the tangent to $\curve(t)$. Similarly, for a vector field $\vv:\surface\to\transl$, where $\transl$ is the translation space associated with $\euclid$,\footnote{Our notation for $\euclid$ and $\transl$ is the same as in \cite[p.\,324]{truesdell:first}, where these geometric structures are further illuminated, especially in connection with their role in formulating  modern continuum mechanics.}
\begin{equation}
	\label{eq:differentiation_vector}
	\frac{\dd}{\dd t}\vv(\curve(t))=(\nablas\vv)\dot{\curve},
\end{equation}
where the second-rank tensor $\nablas\vv$, which annihilates the normal $\normal$, is the surface gradient of $\vv$. In particular, $\nablas\normal$ is the \emph{curvature tensor} of $\surface$, assumed to be at least continuous over $\surface$: it is a symmetric tensor field, whose eigenvalues are the \emph{principal} curvatures of $\surface$.\footnote{Often the curvature tensor is replaced by its opposite, the \emph{shape tensor} $\mathbf{S}=-\nablas\normal$, which clearly encodes the same information. Our choice of sign, which differs from the one customary in differential geometry, is justified by the desire (shared by others \cite{sanders:nonlinear,budiansky:notes}) of designating as $1/R$  the principal curvatures of a sphere of radius $R$. Two other mathematical constructs with essentially the same meaning appear in the literature; these are the \emph{second fundamental form} and the \emph{Weingarten map}, none of which will be used here (see also \cite[p.\,150]{needham:visual} for a similar neglect).}

The surface curl of $\vv$, $\curls\vv$, is defined by the identity
\begin{equation}
	\label{eq:surface_curl}
	2\skw(\nablas\vv)\uv=[\nablas\vv-(\nablas\vv)\trans]\uv=:\curls\vv\times\uv\quad\forall\ \uv\in\transl,
\end{equation}
where $\skw$ denotes the skew-symmetric part of a second-rank tensor, a superscript $\trans$ the transposition of a tensor, and $\times$ the vector product in $\transl$. Equation \eqref{eq:surface_curl} simply says that $\curls\vv$ is the axial vector associated with $2\skw(\nablas\vv)$. Likewise, the surface divergence of $\vv$ is defined as
\begin{equation}
	\label{eq:surface_divergence}
	\divs\vv:=\tr(\nablas\vv),
\end{equation}
where $\tr$ denotes the trace of a second-rank tensor.

As also recalled in \cite{sonnet:bending-neutral}, if $\f$ is a \emph{tangential} vector field, that is, such that $\f\cdot\normal\equiv0$, there exists a scalar field $\varphi$ on $\surface$ such that $\f=\nablas\varphi$, if and only if
\begin{equation}
	\label{eq:integrability_scalar}
	\skw(\nablas\f)=\skw(\curvature\f\otimes\normal).
\end{equation}
Similarly, letting a second-rank tensor field $\F$ be defined on $\surface$ so that $\F\normal\equiv\zero$, there exists a vector field $\vv$ on $\surface$ such that $\F=\nablas\vv$, if and only if
\begin{equation}
	\label{eq:integrability_vector}
	\skw(\nablas\F)=\skw(\F\curvature\otimes\normal),
\end{equation}
where $\nablas\F$ is a third-rank tensor and  skw acts as follows on its generic \emph{triadic} component $\av_1\otimes\av_2\otimes\av_3$,
\begin{equation}
	\label{eq:skw_third_rank_tensor}
	2\skw(\av_1\otimes\av_2\otimes\av_3):=\av_1\otimes\av_2\otimes\av_3-\av_1\otimes\av_3\otimes\av_2.
\end{equation}

It is sometimes convenient to represent a surface $\surface$ by use of coordinates $(u,v)$ as the image of a mapping $\rv:\Omega\to\euclid$, where $\Omega$ is a domain in $\mathbb{R}^2$. For a smooth surface (at least of class $C^2$), coordinates $(u,v)$ can be chosen so as to be, at least locally, \emph{isothermal}, that is, such that 
\begin{equation}
	\label{eq:isothermal_coordinates}
	\rvu\cdot\rvv=0\quad\text{and}\quad|\rvu|=|\rvv|,
\end{equation}
where $\rvu:=\partial_u\rv$ and $\rvv:=\partial_v\rv$. When only the first equation in \eqref{eq:isothermal_coordinates} is satisfied, but not the second, the coordinates are said to be \emph{orthogonal}.\footnote{A classical representation of the Gaussian curvature $K$ in general orthogonal coordinates will be proved in Appendix~\ref{sec:metric} in the vector formalism adopted in this paper.} Letting 
\begin{equation}
	\label{eq:e_u_e_v_definition}
	\e_u:=\frac{\rvu}{|\rvu|}\quad\text{and}\quad\e_v:=\frac{\rvv}{|\rvv|},
\end{equation}
we orient $\surface$ so that $\normal=\e_u\times\e_v$.

\section{Connectors}\label{sec:connectors}
We shall employ the method of \emph{moving frames}.\footnote{This method is systematically presented by  Cartan in \cite{cartan:methode} and is extensively used in the book by O'Neill~\cite{o'neill:elementary} to illustrate the differential geometry of spaces, surfaces, and curves. A more recent, rather comprehensive account, far more formal than needed here, is given in \cite{clelland:from}; a more elementary presentation can be found in \cite{ivey:cartan} (see also footnote~\ref{ftn:history} for earlier precursors).} An orthonormal frame $\frameen$, where $\normal=\e_1\times\e_2$,\footnote{We shall only consider \emph{positively} oriented, orthonormal frames with one unit vector coincident with $\normal$.} glides over $\surface$ according to the laws
\begin{equation}\label{eq:gliding_laws}
	\begin{cases}
		\nablas\e_1=\e_2\otimes\cv+\normal\otimes\dv_1,\\
		\nablas\e_2=-\e_1\otimes\cv+\normal\otimes\dv_2,\\
		\nablas\normal=-\e_1\otimes\dv_1-\e_2\otimes\dv_2,
	\end{cases}
\end{equation}
where the vector fields $(\cv,\dv_1,\dv_2)$ are everywhere tangent to $\surface$; these are the \emph{connectors} of the moving frame. More precisely, $\cv$ is the \emph{spin} connector and $\dv_1$, $\dv_2$, which need not be orthogonal to one another, are the \emph{curvature} connectors.\footnote{We call them connectors because they connect the frame at one point to the frame in a nearby point.} Since the curvature tensor $\nablas\normal$ is symmetric, the curvature connectors must obey the identity
\begin{equation}
	\label{eq:connector_identity}
	\dv_1\cdot\e_2=\dv_2\cdot\e_1.
\end{equation}
In particular, the third equation in \eqref{eq:gliding_laws} implies that
\begin{subequations}\label{eq:mean_and_gaussian}
	\begin{align}
		2H:=&\tr\curvature=-(\dv_1\cdot\e_1+\dv_2\cdot\e_2)\label{eq:mean_curvature},\\
		K:=&\det\curvature=\dv_1\times\dv_2\cdot\normal,\label{eq:gaussian_curvature}
	\end{align}
\end{subequations}
where $H$ and $K$ are the \emph{mean} and \emph{Gaussian} curvatures of $\surface$, respectively.\footnote{By $\tr\curvature$ and $\det\curvature$, we mean the sum and the product of the principal curvatures of $\surface$, respectively.}
Below, we shall apply \eqref{eq:gliding_laws} to a variety of different moving frames, each entailing its own set of connectors.

We now elaborate further upon the integrability condition \eqref{eq:integrability_vector} by applying it to the tensor fields $\nablas\e_1$ and $\nablas\e_2$. We first let $\F=\nablas\e_1$. Repeated use of \eqref{eq:gliding_laws} easily shows that
\begin{equation}
	\label{eq:nabla_2_e_1}
	\nablastwo\e_1=-\e_1\otimes\cv\otimes\cv+\normal\otimes\cv\otimes\dv_2+\e_2\otimes\nablas\cv-\e_1\otimes\dv_1\otimes\dv_1+\normal\otimes\nablas\dv_1
\end{equation}
and
\begin{equation}
	\label{eq:nabla_2_e_1_associate}
	(\nablas\e_1)\curvature\otimes\normal=-c_1\e_2\otimes\dv_1\otimes\normal-c_2\e_2\otimes\dv_2\normal-d_{11}\normal\otimes\dv_1\otimes\normal-d_{12}\normal\otimes\dv_2\otimes\normal.
\end{equation}
Here and below, we set 
\begin{equation}
	\label{eq:c_d_1_components}
	\cv=c_1\e_1+c_2\e_2,\ \dv_1=d_{11}\e_1+d_{12}\e_2,\ \dv_2=d_{21}\e_1+d_{22}\e_2,\quad \text{with}\quad  d_{12}=d_{21}.
\end{equation}

By applying \eqref{eq:integrability_vector} to the pair of equations \eqref{eq:nabla_2_e_1}, \eqref{eq:nabla_2_e_1_associate} and their analogs for $\F=\nablas\e_2$, we obtain formulas that can be written in  the following general form,
\begin{equation}
	\label{eq:integrability_vetor_representation}
	\skw\nablas\F-\skw(\F\curvature\otimes\normal)=\e_1\otimes\W_1+\e_2\otimes\W_2+\normal\otimes\W_3,
\end{equation}
where $\W_1$, $\W_2$, and $\W_3$ are skew second-rank tensors that must all vanish for \eqref{eq:integrability_vector} to hold, since $\frameen$ is a frame.
Once written for the axial vectors associated with these skew-symmetric tensors, the corresponding identities become\footnote{No further information could be obtained by requiring the integrability of $\nablas\normal$, as this would follow from \eqref{eq:gauss_mainardi} and the fact that $\normal=\e_1\times\e_2$.}
\begin{equation}
	\label{eq:integrability_identities}
	\begin{cases}
		\curls\cv+\dv_1\times\dv_2=c_1\dv_1\times\normal+c_2\dv_2\times\normal,\\
		\curls\dv_1-\cv\times\dv_2=d_{11}\dv_1\times\normal+d_{12}\dv_2\times\normal,\\
		\curls\dv_2+\cv\times\dv_1=d_{21}\dv_1\times\normal+d_{22}\dv_2\times\normal,
	\end{cases}
\end{equation}
which link all connectors together. A telling consequence of these equations follows from projecting along $\normal$ both their sides: also by use of \eqref{eq:gaussian_curvature}, we arrive at 
\begin{subequations}\label{eq:gauss_mainardi}
	\begin{align}
		\curls\cv\cdot\normal&=-K,\label{eq:gauss_equation}\\
		\curls\dv_1\cdot\normal&=\cv\times\dv_2\cdot\normal,\label{eq:mainardi_equation_1}\\
		\curls\dv_2\cdot\normal&=-\cv\times\dv_1\cdot\normal.	\label{eq:mainardi_equation_2}
	\end{align}
\end{subequations}

Equations \eqref{eq:gauss_mainardi} are fundamental in the differential geometry of surfaces in three-dimensional space. They correspond to the equations that in Cartan's language of differential forms  are derived in \cite{o'neill:elementary} (see claims (3) and (4) of Theorem~1.7, p.\,267) or \cite[p.\,448]{needham:visual}. Both \cite{o'neill:elementary} and \cite{needham:visual} call \eqref{eq:gauss_equation} the Gauss equation; \cite{o'neill:elementary} calls \eqref{eq:mainardi_equation_1} and \eqref{eq:mainardi_equation_2} the Codazzi equations, whereas \cite{needham:visual}, in an attempt to establish proper priority, calls them the Peterson-Mainardi-Codazzi equations.\footnote{The order of names would reflect the chronology of the discovery. Peterson apparently obtained these equations in his 1853 Thesis (originally written in Latvian), which was translated to Russian and published in 1952 \cite{phillips:karl}. In 1856, Mainardi found the same equations, which were then rediscovered independently by Codazzi in 1860  \cite[see also][p.\,448]{needham:visual}.}

Equation \eqref{eq:gauss_equation}, a further classical consequence of which is illustrated in Appendix~\ref{sec:metric}, establishes a relation between the spin connector $\cv$ of any moving frame and the Gaussian curvature $K$ of $\surface$, a quantity pertaining to the \emph{intrinsic} geometry of the surface.\footnote{As effectively recalled in \cite[p.\,11]{needham:visual},
``[I]ntrinsic geometry of a surface [is] a fundamentally new view of geometry, introduced by Gauss~\cite{gauss:general}. It means the geometry that is knowable to tiny, ant-like, intelligent (but 2-dimensional) creatures living \emph{within} the surface.''}
Equations \eqref{eq:mainardi_equation_1} and \eqref{eq:mainardi_equation_2} relate the curvature connectors of a moving frame to the spin connector of the same frame.

Connectors transform in a simple way under a change of moving frame. Consider a frame $\frameprime$ related to $\frameen$ through the equations 
\begin{equation}
	\label{eq:starred_frame}
	\e_1^\prime=\cos\alpha\e_1+\sin\alpha\e_2\quad\e_2^\prime=-\sin\alpha\e_1+\cos\alpha\e_2,
\end{equation}
where $\alpha$ is a scalar function of class $C^2$ on $\surface$. The frame $\frameprime$ obeys the same gliding laws \eqref{eq:gliding_laws}, but with primed connectors $(\cv^\prime,\dv_1^\prime,\dv_2^\prime)$. To relate these to $(\cv,\dv_1,\dv_2)$, we differentiate both sides of \eqref{eq:starred_frame} and make use of the resulting equations in the primed version of \eqref{eq:gliding_laws}. By identification of terms, we conclude that $\cv$ is shifted into
\begin{equation*}
	\label{eq:c_star_shift}
	\cv^\prime=\cv+\nablas\alpha,
\end{equation*}
while $\dv_1$ and $\dv_2$ transform like $\e_1$ and $\e_2$ in \eqref{eq:starred_frame},
\begin{equation}
	\label{eq:starred_connectors}
	\dv_1^\prime=\cos\alpha\dv_1+\sin\alpha\dv_2,\quad\dv_2^\prime=-\sin\alpha\dv_1+\cos\alpha\dv_2.
\end{equation}
Since, by \eqref{eq:integrability_scalar}, $\curls(\nablas\alpha)\cdot\normal\equiv0$, we immediately see that $\cv^\prime$ satisfies \eqref{eq:gauss_equation} like $\cv$. It is also easy to show from \eqref{eq:starred_connectors} that $\dv_1^\prime$ and $\dv_2^\prime$ obey the symmetry relation \eqref{eq:connector_identity} whenever $\dv_1$ and $\dv_2$ do, irrespective of the choice of the rotation angle $\alpha$. Moreover, $|\dv_1^\prime|^2+|\dv_2^\prime|^2=|\dv_1|^2+|\dv_2|^2$, as it should, since $\tr\curvature^2$ is the same for both frames.

A special choice of moving frame gives \eqref{eq:mainardi_equation_1} and \eqref{eq:mainardi_equation_2} another, possibly more transparent form. Let $\frameen$ be chosen as the eigenframe of the curvature tensor $\nablas\normal$, so that
\begin{equation}
	\label{eq:curvature_diagonal}
	\nablas\normal=\kappa_1\e_1\otimes\e_1+\kappa_2\e_2\otimes\e_2,
\end{equation}
where $\kappa_1$ and $\kappa_2$ are the principal curvatures of $\surface$. By comparing \eqref{eq:curvature_diagonal} and \eqref{eq:gliding_laws}, we see that in this frame
\begin{equation}
	\label{eq:d_connectors_diagonal}
	\dv_1=-\kappa_1\e_1\quad\text{and}\quad\dv_2=-\kappa_2\e_2.
\end{equation}
It follows by direct computation from the first equation in \eqref{eq:d_connectors_diagonal} that
\begin{align}
	\label{eq:curl_d_1_diagonal}
	\curls\dv_1\cdot\normal&=-\kappa_1\curls\e_1\cdot\normal+\e_2\cdot\nablas\kappa_1\nonumber\\
	&=-\kappa_1\e_1\cdot\cv+\e_2\cdot\nablas\kappa_1,
\end{align}
where use has also been made of the first equation in \eqref{eq:gliding_laws}.\footnote{A general identity for the mixed product, $\av\times\bv\cdot\cv=\bv\times\cv\cdot\av=\cv\times\av\cdot\bv$, valid for any triple of vectors $(\av,\bv,\cv)$, is tacitly implied in \eqref{eq:curl_d_1_diagonal} and elsewhere below.} By inserting in \eqref{eq:mainardi_equation_1} the second equation in \eqref{eq:d_connectors_diagonal} we arrive at
\begin{subequations}
	\label{eq:mainardi_special}
	\begin{equation}
		\label{eq:mainardi_special_1}
		\e_2\cdot\nablas\kappa_1=(\kappa_1-\kappa_2)(\e_1\cdot\cv).
	\end{equation}
	Similarly, we give \eqref{eq:mainardi_equation_2} the form\footnote{See also \cite[\S\,38.6, especially p.\,450]{needham:visual} for a proof of \eqref{eq:mainardi_special} phrased in the (equivalent) language of differential forms.}
	\begin{equation}
		\label{eq:mainardi_special_2}
		\e_1\cdot\nablas\kappa_2=(\kappa_1-\kappa_2)(\e_2\cdot\cv).
	\end{equation}
\end{subequations}

\section{Kinematics of surfaces}\label{sec:kinematics}
In an attempt to build a direct theory for the elasticity of \emph{soft} shells,\footnote{These are the ones expected to be more prone than others to in-plane twisting, possibly relevant in biophysical applications.} we have introduced \emph{drilling} and \emph{bending contents} in the kinematics of material surfaces \cite{sonnet:bending-neutral}. Here, we recall these definitions along with the corresponding invariant measures (of drilling and bending) that were attributed in \cite{sonnet:variational} to separate deformation modes.

To this end, we first collect a number of general results about the kinematics of surfaces. Consider a deformation $\y:\surface\to\euclid$ of a smooth surface $\surface$, which maps $\surface$ into $\surface^\ast=\y(\surface)$. Let $\y$ be twice continuously differentiable. By the polar decomposition theorem for the deformations of surfaces \cite{man:coordinate}, we can write
\begin{equation}
	\label{eq:deformation_gradient}
	\nablas\y=\R\U=\V\R,
\end{equation}
where, for every $\x\in\surface$, $\R(\x)$ is a rotation of the special orthogonal group $\orth$ in three-dimensional translation space $\transl$, $\U(\x)$ is a symmetric second-rank tensor on the tangent plane $\tplane_{\x}$ to $\surface$ at $\x$, whereas $\V(\y(\x))$ is a symmetric second-rank tensor on the tangent plane $\tplane^\ast_{\y(\x)}$ to $\surface^\ast$ at $\y(\x)$.

 We shall also refer to $\R$ as the polar rotation of $\nablas\y$; its determinant is $+1$ because the deformation $\y$ is required to preserve the \emph{orientation} of $\surface$. The tensors $\U$ and $\V$ are the \emph{stretching} tensors; they are both positive definite and can be represented as 
\begin{equation}
	\label{eq:U_V}
	\U=\lambda_1\uv_1\otimes\uv_1+\lambda_2\uv_2\otimes\uv_2\quad\text{and}\quad \V=\lambda_1\vv_1\otimes\vv_1+\lambda_2\vv_2\otimes\vv_2,
\end{equation}
where $\lambda_1,\lambda_2>0$ are the \emph{principal stretches} and the pairs $(\uv_1,\uv_2)$ and $(\vv_1,\vv_2)$ the corresponding \emph{principal directions} of stretching on $\tplane_{\x}$ and $\tplane^\ast_{\y(\x)}$, respectively. It readily follows from \eqref{eq:deformation_gradient} that
\begin{equation}
	\label{eq:V_v}
	\V=\R\U\R\trans\quad\text{and}\quad\vv_i=\R\uv_i\quad\text{for}\quad i=1,2.
\end{equation}
A useful representation of $\R$ is thus given by
\begin{equation}
	\label{eq:R_representation}
	\R=\vv_1\otimes\uv_1+\vv_2\otimes\uv_2+\normal^\ast\otimes\normal,
\end{equation}
where $\normal^\ast$ is the unit normal to $\surface^\ast$ oriented coherently with $\normal$.\footnote{That is, so that the frames $\frameu$ and $\framev$ are equally oriented, $\uv_1\times\uv_2\cdot\normal=\vv_1\times\vv_2\cdot\normal^\ast$.}

The \emph{right} and \emph{left} Cauchy-Green tensors embodying the metric properties of $\surface^\ast$ compared with $\surface$ (and of $\surface$ compared with $\surface^\ast$), acting on $\tplane_{\x}$ and $\tplane^\ast_{\y(\x)}$, respectively, are defined as
\begin{equation}
	\label{eq:C_B_tensors}
	\C:=(\nablas\y)\trans(\nablas\y)\quad\text{and}\quad\B:=(\nablas\y)(\nablas\y)\trans.
\end{equation}
Both $\C$ and $\B$ are  symmetric, positive definite tensors.

As discussed in \cite{sonnet:bending-neutral}, by applying Rodrigues' formula for representing rotations \cite{rodrigues:lois}, $\R$ can be uniquely decomposed into a \emph{drilling} rotation $\Rd$ and a \emph{bending} rotation $\Rb$,
\begin{equation}
	\label{eq:R_decomposition}
	\R=\Rb\Rd\quad\text{with}\quad\Rd\in\mathsf{SO}(\normal),\ \Rb\in\mathsf{SO}(\e),\ \e\cdot\normal=0,
\end{equation}
where $\mathsf{SO}(\e)$ is the subgroup of $\orth$ of all rotations about the axis designated by the unit vector $\e$.\footnote{Formally, $\mathsf{SO}(\e):=\{\R\in\orth:\R\e=\e\}$.} Thus, $\Rd$ is a rotation about the normal $\normal$ to $\surface$, while $\Rb$ is a rotation about an axis on the tangent plane (uniquely determined by $\R$). The vectors $\dv$ and $\bv$ that represent $\Rd$ and $\Rb$ via the Rodrigues formula, the former parallel to $\normal$ and the latter orthogonal to it, are recorded for completeness in Appendix~\ref{sec:contents}.

We say that in a given frame a deformation is of \emph{pure drilling} or of \emph{pure bending} if $\Rb=\I$ or $\Rd=\I$, respectively. Alternatively, we call the former \emph{bending-neutral} and the latter \emph{drilling-neutral}, as they have the property of leaving the bending or drilling content unaltered when composed with a pre-existing deformation. Both bending-neutral and drilling-neutral deformations of \emph{minimal surfaces}\footnote{These are surfaces with zero mean curvature $H$.} have been investigated in \cite{sonnet:neutral}.

Here we are interested in seeing what implications some metric restrictions have on these classes of deformations.

In a general deformation $\y$ of $\surface$, we can identify three  independent modes: they are \emph{stretching}, \emph{drilling}, and \emph{bending}, to which there correspond three \emph{pure} deformation measures, which we denote $\wst$, $\wdr$, and $\wbe$, respectively.\footnote{A pure measure of deformation is an invariant measure that is selectively activated when there is a frame where  the polar rotation $\R$ of $\nablas\y$ is either $\I$, $\Rd$, or $\Rb$, see \cite{sonnet:variational}.} For $\wst$ we take 
\begin{equation}
	\label{eq:w_s}
	\wst:=|\U-\proj|^2=|\V-\mathbf{P}(\normal^\ast)|^2,
\end{equation}
which measures the metric mismatch between $\surface$ and $\surface^\ast$.\footnote{The energy $\wst$ as delivered by \eqref{eq:w_s} is quadratic in the principal stretches; other forms have been proposed for $\wst$, such as Koiter's \cite{koiter:consistent}, which is \emph{quartic} in the principal stretches. Here, $\wst$ will play only a small role, as we are mainly interested in isomteric deformations of $\surface$.} Both $\wdr$ and $\wbe$ were defined in \cite{sonnet:variational} by taking invariant averages of the scalar contents $d^2$ and $b^2$, respectively, which delivered\footnote{It is perhaps worth noting that $\wdr$ is quadratic in $\H$, whereas $\wbe$ is quartic.}
\begin{equation}
	\label{eq:w_d_w_b}
	\wdr:=|\W(\normal)\circ\H|^2\quad\text{and}\quad\wbe:=\left(|\H|^2-\frac12|\W(\normal)\circ\H|^2-4\normal\cdot\H\circ\nablas\normal\right)^2.
\end{equation}
Here $\H$ is the third-rank tensor defined as\footnote{The tensor $\H$ is somewhat reminiscent of the third-rank tensor $\G:=\F\trans\nablas\F$, where $\F=\nablas\y$. The tensor $\G$ was introduced  by Murdoch~\cite{murdoch:direct}  in  his direct second-grade hyperelastic theory of shells and  has recently received a full kinematic characterization in \cite{tiwari:characterization}.} 
\begin{equation}
	\label{eq:H_definition}
	\H:=\R\trans\nablas\R
\end{equation}
and two vector-valued products by a third- and a second-rank tensor have been introduced, both denoted by $\circ$ and distinguished only by the order of multiplication, that in a given basis $\framec$ are represented in component form as\footnote{Here and in the following, the standard convention of summing over repeated indices is adopted.}
\begin{equation}
\label{eq:multipication_rules}
\A\circ\H=A_{ij}H_{ijk}\e_k\quad\text{and}\quad\H\circ\A=H_{ijk}A_{jk}\e_i.
\end{equation}
For $\A$ and $\H$ in the special diadic and triadic forms $\A=\av_1\otimes\av_2$ and $\H=\bv_1\otimes\bv_2\otimes\bv_3$,\footnote{Incidentally, for such an $\H$, $|\H|^2=b_1^2b_2^2b_3^2$.}
\begin{equation}
	\label{eq:multiplication_rules_examples}
	\A\circ\H=(\av_1\cdot\bv_1)(\av_2\cdot\bv_2)\bv_3\quad\text{and}\quad\H\circ\A=(\av_1\cdot\bv_2)(\av_2\cdot\bv_3)\bv_1.
\end{equation}

The symbol $\circ$ will also be employed with yet a different meaning in the generation of a second-rank tensor as product of a third-rank tensor, say $\H$, by a vector $\av$,
\begin{equation}
\label{eq:H_a}
\H\circ\av:=H_{ikj}a_k\e_i\otimes\e_j.
\end{equation}
No confusion should arise from our use of $\circ$ in \eqref{eq:H_a}, as this involves a vector $\av$, whereas a second-rank tensor $\A$ is involved in \eqref{eq:multipication_rules}.

\subsection{Invariant rotation gradient}\label{sec:H_tensor}
The tensor $\H$ will play a central role in our development. In component form, it reads explicitly as
\begin{equation}
	\label{eq:H_components}
	\H=R_{hi}R_{hj;k}\e_i\otimes\e_j\otimes\e_k,
\end{equation} 
where a semicolon denotes surface differentiation.

We shall call $\H$ the \emph{invariant rotation gradient}, for the reason that we now explain. A change of frame is represented by a rotation $\Q\in\orth$, uniform in space and possibly depending on time \cite[p.\,325]{truesdell:first}. The polar rotation $\R$ of a deformation gradient is transformed in $\R^\ast=\Q\R$ by a change of frame, and so it is neither \emph{frame-indifferent} nor \emph{frame-invariant}.\footnote{A frame-indifferent second-rank  tensor $\G$ would transform as $\G^\ast=\Q\G\Q\trans$, while a frame-invariant one as $\G^\ast=\G$ (see \cite[p.\,150]{gurtin:mechanics}).}

Since $\R\in\orth$, and so $\R\trans\R=\I$, \eqref{eq:H_components} readily entails that  the invariant rotation gradient can be represented as follows,
\begin{equation}
	\label{eq:H_representation}
	\H=\W(\uv_1)\otimes\av_1+\W(\uv_2)\otimes\av_2+\W(\normal)\otimes\av_3,
\end{equation}
in terms of three \emph{tangential} vectors $\av_1$, $\av_2$, and $\av_3$, see also \cite{sonnet:variational}. It should be noted that these vectors depend on the choice of the moving frame on $\surface$; equation \eqref{eq:H_representation} applies to the frame $\frameu$ of principal directions of stretching.

Below we collect a number of properties of $\H$, which shall be useful in the rest of the paper. The gliding laws in \eqref{eq:gliding_laws} can be written as follows for both frames $\frameu$ and $\framev$ on $\surface$ and $\surface^\ast$, respectively,
\begin{equation}\label{eq:connectors}
	\begin{cases}
		\nablas\uv_1=\uv_2\otimes\cv+\normal\otimes\dv_1,\\
		\nablas\uv_2=-\uv_1\otimes\cv+\normal\otimes\dv_2,\\
		\nablas\normal=-\uv_1\otimes\dv_1-\uv_2\otimes\dv_2,
	\end{cases}
	\qquad
	\begin{cases}
		\nablast\vv_1=\vv_2\otimes\cv^\ast+\normal^\ast\otimes\dv_1^\ast,\\
		\nablast\vv_2=-\vv_1\otimes\cv^\ast+\normal^\ast\otimes\dv_2^\ast,\\
		\nablast\normal^\ast=-\vv_1\otimes\dv_1^\ast-\vv_2\otimes\dv_2^\ast,
	\end{cases}
\end{equation}
where the connectors $(\cv,\dv_1,\dv_2)$ are associated with the frame $\frameu$, while $(\cv^\ast,\dv_1^\ast,\dv_2^\ast)$ are associated with $\framev$, and $\nablast$ denotes the surface gradient on $\surface^\ast$. 

It was proved in \cite{sonnet:variational} that with the aid of \eqref{eq:connectors} the pure measures of drilling and bending in \eqref{eq:w_d_w_b} can be given the form
\begin{subequations}
	\label{eq:w_d_w_b_connectors}
\begin{align}
	\wdr&=4|\V\cv^\ast-\R\cv|^2,\label{eq:wd_connectors}\\
	\wbe&=4\left[|\V(\nablast\normal^\ast)|^2-|\nablas\normal|^2\right]^2=4\left[\lambda_1^2(d_1^\ast)^2+\lambda_2^2(d_2^\ast)^2-d_1^2-d_2^2\right]^2.\label{eq:wb_connectors}
\end{align}
\end{subequations}

By \eqref{eq:V_v} and the chain rule, 
\begin{equation}
	\label{eq:nablas_v_1}
	\nablas\vv_1=\R\nablas\uv_1+\nablas\R\circ\uv_1=(\nablast\vv_1)\nablas\y.
\end{equation}
Making use of both \eqref{eq:H_definition} and the second set of equations in \eqref{eq:connectors}, we give \eqref{eq:nablas_v_1} the following form,
\begin{equation}
	\label{eq:compatibility_nablas_v_1}
	\uv_2\otimes\cv+\normal\otimes\dv_1+\H\circ\uv_1=\uv_2\otimes\U\R\trans\cv^\ast+\normal\otimes\U\R\trans\dv_1^\ast,
\end{equation}
where \eqref{eq:R_representation} has also been employed. Since $(\uv_1,\uv_2)$ is a basis of $\tplane_{\x}$ and, as can easily been shown, $(\H\circ\uv_1)\trans\uv_1=\zero$, equation \eqref{eq:compatibility_nablas_v_1} is equivalent to the pair
\begin{subequations}
	\label{eq:a_3_a_2_a_1}
	\begin{align}
		\V\cv^\ast-\R\cv&=\R\av_3,\label{eq:a_3}\\
		\V\dv_1^\ast-\R\dv_1&=-\R\av_2,\label{eq:a_2}
	\end{align}
where the following identities, which are immediate consequences of \eqref{eq:H_representation}, have also been used,
\begin{equation*}
	\label{eq:a_3_a_2_identities}
	(\H\circ\uv_1)\trans\uv_2=\av_3,\quad(\H\circ\uv_1)\trans\normal=-(\H\circ\normal)\trans\uv_1=-\av_2.
\end{equation*}
By reasoning similarly with the decomposition of $\nablas\vv_2$ that parallels \eqref{eq:nablas_v_1}, we supplement \eqref{eq:a_3} and \eqref{eq:a_2} with\footnote{Equations \eqref{eq:a_2} and \eqref{eq:a_1} can also be obtained by decomposing $\nablas\normal^\ast$  in a fashion similar to \eqref{eq:nablas_v_1}.}
\begin{equation}
	\label{eq:a_1}
	\V\dv_2^\ast-\R\dv_2=\R\av_1.
\end{equation}
\end{subequations}
The three equations \eqref{eq:a_3_a_2_a_1} give $\wdr$ and $\wbe$ a new form in terms of the three vectors $\av$,
\begin{equation}
	\label{eq:wd_wb_a}
	\wdr=4a_3^2,\quad\wbe=4\left(a_1^2+a_2^2+2(\av_1\cdot\dv_2-\av_2\cdot\dv_1) \right)^2.
\end{equation}

Two further properties of $\H$ will be of some use; we close this section by illustrating them.

First, as shown in detail in Appendix~\ref{sec:integrability}, the integrability condition \eqref{eq:integrability_vector} applied to $\F=\nablas\y$ reduces to the following system of scalar equations,
\begin{subequations}\label{eq:integrability_conditions}
\begin{align}[left={\empheqlbrace}]
	&\lambda_1(\dv_1-\av_2)\cdot\uv_2=\lambda_2(\dv_2+\av_1)\cdot\uv_1, \label{eq:integrability_condition_1}\\
	&\lambda_1(\av_3\cdot\uv_2)+(\lambda_1-\lambda_2)(\cv\cdot\uv_2)-\nablas\lambda_2\cdot\uv_1=0, \label{eq:integrability_condition_2}\\
	&\lambda_2(\av_3\cdot\uv_1)+(\lambda_2-\lambda_1)(\cv\cdot\uv_1)+\nablas\lambda_1\cdot\uv_2=0.\label{eq:integrability_condition_3}
\end{align}
\end{subequations}

Second, by expanding $\dv_1$ and $\dv_2$ in the frame $\frameu$ as done in \eqref{eq:c_d_1_components} for the generic frame $\frameen$, from \eqref{eq:a_2} and \eqref{eq:a_1} we derive a similar expansion for $\dv_1^\ast$ and $\dv_2^\ast$ in the frame $\framev$,
\begin{subequations}
\label{eq:d_1_ast_d_2_ast}
\begin{align}
	\dv_1^\ast&=\frac{1}{\lambda_1}(d_{11}-a_{21})\vv_1+\frac{1}{\lambda_2}(d_{12}-a_{22})\vv_2,\label{eq:d_1_ast}\\
	\dv_2^\ast&=\frac{1}{\lambda_1}(d_{21}+a_{11})\vv_1+\frac{1}{\lambda_2}(d_{22}+a_{12})\vv_2,\label{eq:d_2_ast}
\end{align}
\end{subequations}
where we have set $\av_1=a_{11}\uv_1+a_{12}\uv_2$ and $\av_2=a_{21}\uv_1+a_{22}\uv_2$. Use of \eqref{eq:integrability_condition_1} in \eqref{eq:d_1_ast_d_2_ast} makes sure that the connectors $(\dv_1^\ast,\dv_2^\ast)$ obey the symmetry condition \eqref{eq:connector_identity}.

Combining \eqref{eq:d_1_ast_d_2_ast} and \eqref{eq:gaussian_curvature}, we readily arrive at the following formula for the Gaussian curvature $K^\ast$ of $\surface^\ast$, where the $\av$ vectors also feature explicitly,
\begin{equation}
	\label{eq:K_ast}
	K^\ast=\frac{1}{\det\U}[K-(\av_1\times\av_2-\av_1\times\dv_1-\av_2\times\dv_2)\cdot\normal],
\end{equation}
where $K$ is the Gaussian curvature of $\surface$.

In the next sections, our development will build on the preliminary material collected in this section and the preceding one; to make our results easier to retrace, they will be presented in a formal, more structured way.


\section{Conformal deformations}\label{sec:conformal}
Here we see how the properties of the invariant rotation gradient outlined above relate to those special deformations of $\surface$ that preserve \emph{angles}. Such deformations are called \emph{conformal} and are formally defined as follows.
\begin{definition}\label{def:conformal}
A deformation $\y$ is \emph{conformal} if
\begin{equation}
\label{eq:conformal_definition}
\frac{\uv\cdot\vv}{|\uv||\vv|}=\frac{(\nablas\y)\uv\cdot(\nablas\y)\vv}{|(\nablas\y)\uv||(\nablas\y)\vv|},\quad\forall\ \uv,\vv\in\tplane_{\x}\setminus\{\zero\}.
\end{equation}
\end{definition} 
\begin{remark}
	\label{rmk:conformal_C}
	It readily follows from \eqref{eq:C_B_tensors} that \eqref{eq:conformal_definition} can equivalently be rewritten as
	\begin{equation}
		\label{eq:conformal_rewritten}
		\frac{\uv\cdot\vv}{|\uv||\vv|}=\frac{\uv\cdot\C\vv}{(\uv\cdot\C\uv)^{1/2}(\vv\cdot\C\vv)^{1/2}}\quad\forall\ \uv,\vv\in\tplane_{\x}\setminus\{\zero\}.
	\end{equation}
\end{remark}
\begin{proposition}
	\label{prop:conformal_C}
	A deformation $\y$ is conformal if and only if
	\begin{equation}
	\label{eq:conformal_characterization}
	\C=\lambda^2\proj,
	\end{equation}
	where $\lambda>0$ is a scalar field on $\surface$.
\end{proposition}
\begin{proof}
	One implication is trivial as, for $\C$ as in \eqref{eq:conformal_characterization}, \eqref{eq:conformal_rewritten} becomes an identity. To prove that \eqref{eq:conformal_rewritten} implies \eqref{eq:conformal_characterization}, we first remark that \eqref{eq:conformal_rewritten} specializes into
	\begin{equation}
	\label{eq:conformal_orthogonal}
	\uv\cdot\C\vv=0\quad\text{for}\quad\uv\cdot\vv=0.
	\end{equation}
	Then we resort to \eqref{eq:U_V} and write
	\begin{equation}
		\label{eq:conformal_C}
		\C=\lambda_1^2\uv_1\otimes\uv_1+\lambda_2^2\uv_2\otimes\uv_2.
	\end{equation}
	By use of \eqref{eq:conformal_C} in \eqref{eq:conformal_orthogonal} with $\uv=\cos\vt\uv_1+\sin\vt\uv_2$ and $\vv=-\sin\vt\uv_1+\cos\vt\uv_2$, we readily arrive at
	\begin{equation}
	\label{eq:conformal_orthogonal_theta}
	\uv\cdot\C\vv=(\lambda_2^2-\lambda_1^2)\sin\vt\cos\vt=0\quad\forall\ \vt,
	\end{equation}
	which requires $\lambda_1^2=\lambda_2^2=\lambda^2$, and so implies \eqref{eq:conformal_characterization}.
\end{proof}
\begin{remark}
	\label{rmk:conformal_alternative}
	An alternative proof of Proposition~\ref{prop:conformal_C} was also obtained in \cite{seguin:coordinate} by a different reasoning.
\end{remark}
\begin{remark}
	\label{rmk:conformal_U_V}
	As a consequence of \eqref{eq:conformal_C} and \eqref{eq:U_V}, for a conformal deformation the stretching tensors read as follows,
	\begin{equation}
		\label{eq:conformal_U_V}
		\U=\lambda\proj\quad\text{and}\quad\V=\lambda\mathbf{P}(\normal^\ast).
	\end{equation}
\end{remark}
We now determine the $\av$ vectors that represent the invariant rotation gradient $\H$ for a conformal deformation.
\begin{remark}
	\label{rmk:conformal_a}
	Since for a conformal deformation the principal stretches are equal, $\lambda_1=\lambda_2=\lambda>0$, under the assumption that $\lambda$ be differentiable, it follows from \eqref{eq:integrability_condition_2} and \eqref{eq:integrability_condition_3} that 
	\begin{subequations}
		\label{eq:conformal_a}
		\begin{equation}
		\label{eq:conformal_a_3}
		\av_3=\normal\times\nablas\ln\lambda.
		\end{equation}
		Similarly, by use of \eqref{eq:conformal_U_V} in \eqref{eq:a_2} and \eqref{eq:a_1} we easily arrive at
		\begin{align}
			\av_2&=-\lambda\R\trans\dv_1^\ast+\dv_1,\label{eq:conformal_a_2}\\
			\av_1&=\lambda\R\trans\dv_2^\ast-\dv_2.\label{eq:conformal_a_1}
		\end{align}
	\end{subequations}
\end{remark}
\begin{remark}
	\label{rmk:conformal_c_ast}
	By combining \eqref{eq:conformal_a_3} and \eqref{eq:a_3}, we also write the spin connector $\cv^\ast$ associated on $\surface^\ast$ with the frame $\framev$ as
	\begin{equation}\label{eq:conformal_h}
		\cv^\ast=\R\h\quad\text{with}\quad\h:=\frac{1}{\lambda}\left(\cv+\normal\times\nablas\ln\lambda \right).
	\end{equation} 
\end{remark}
\begin{remark}
	\label{rmk:conformal_a_1_a_2}
	For a conformal deformation, the integrability condition \eqref{eq:integrability_condition_1} combined with the symmetry property \eqref{eq:connector_identity}  delivers
	\begin{equation}
	\label{eq:conformal_a_1_a_2}
	\av_1\cdot\e_1+\av_2\cdot\e_2=0,
	\end{equation}
	valid in any  frame $\frameen$ on $\surface$ with the corresponding choice of the $\av$ vectors.
\end{remark}
\begin{remark}
	\label{rmk:conformal_mean_curvature}
	It follows from \eqref{eq:mean_curvature}, \eqref{eq:d_connectors_diagonal}, and \eqref{eq:conformal_a_2}, \eqref{eq:conformal_a_1} that a conformal deformation transforms the mean curvature $H$ of $\surface$ into the mean curvature $H^\ast$ of $\surface^\ast$ according to the law,
	\begin{equation}
		\label{eq:conformal_mean_curvature}
		2H^\ast=\frac{1}{\lambda}\left(2H+\av_2\cdot\e_1-\av_1\cdot\e_2\right).
	\end{equation} 
\end{remark}
We are now in a position to relate the Gaussian curvature $K^\ast$ of $\surface^\ast$ to the Gaussian curvature $K$ of $\surface$ whenever the deformation $\y$ is conformal.
\begin{proposition}
	\label{prop:conformal_gaussian_curvature}
	For a conformal deformation $\y$ of $\surface$ into $\surface^\ast$,
	\begin{equation}
		\label{eq:conformal_gaussian_curvature}
		K^\ast=\frac{1}{\lambda^2}\left(K-\laps\ln\lambda\right),
	\end{equation}
	where $\laps:=\divs\nablas$ denotes the \emph{surface} Laplacian.
\end{proposition}
\begin{proof}
By \eqref{eq:gauss_equation} applied to $\surface^\ast$, $K^\ast=-\curls\cv^\ast\cdot\normal^\ast$, and so we need first compute $\nablast\cv^\ast$. To this end, we remark that, by the chain rule,
\begin{equation}
	\label{eq:confornal_nablas_c_ast}
	\nablas\cv^\ast=(\nablast\cv^\ast)\nablas\y=\lambda(\nablast\cv^\ast)\mathbf{P}(\normal^\ast)\R,
\end{equation}
which, also by \eqref{eq:conformal_h}, yields
\begin{equation}
\label{eq:conformal_nablast_c_ast}
\nablast\cv^\ast=\frac{1}{\lambda}\nablas(\R\h)\R\trans=\frac{1}{\lambda}\left(\R(\nablas\h)\R\trans+(\nablas\R\circ\h)\R\trans\right).
\end{equation}
To compute $\nablas\R$ for $\R$ expressed as in \eqref{eq:R_representation}, we need the following formulae,
\begin{equation}
	\label{eq:nablas_v_s}
	\begin{cases}
		\nablas\vv_1=\lambda\vv_2\otimes\h+\normal^\ast\otimes\R\trans\dv_1^\ast,\\
		\nablas\vv_2=-\lambda\vv_1\otimes\h+\normal^\ast\otimes\R\trans\dv_2^\ast,\\
		\nablas\normal^\ast=-\lambda\vv_1\otimes\R\trans\dv^\ast_1-\lambda\vv_2\otimes\R\trans\dv_2^\ast,
	\end{cases}
\end{equation}
which are obtained by applying the chain rule and \eqref{eq:conformal_h} to the second set of equations in \eqref{eq:connectors}. By use of \eqref{eq:nablas_v_s}, the explicit form of $\h$ in \eqref{eq:conformal_h}, and both \eqref{eq:conformal_a_2} and \eqref{eq:conformal_a_1}, we obtain that 
\begin{equation}
	\label{eq:conformal_nablas_R}
	\begin{split}
	\nablas\R&=\vv_1\otimes(\normal\otimes\av_2-\uv_2\otimes\normal\times\nablas\ln\lambda)-\vv_2\otimes(\normal\otimes\av_1-\uv_1\otimes\normal\times\nablas\ln\lambda)\\
	&+\normal^\ast\otimes(\uv_2\otimes\av_1-\uv_1\otimes\av_2).
	\end{split}
\end{equation}
It is then a simple matter to arrive at
\begin{equation}\label{eq:conformal_intermediate}
	\begin{split}
		(\nablas\R\circ\h)\R\trans&=(\h\cdot\uv_1)[\R(\uv_2\otimes\normal\times\nablas\ln\lambda)\R\trans-\normal^\ast\otimes\R\av_2]\\
		&-(\h\cdot\uv_2)[\R(\uv_1\otimes\normal\times\nablas\ln\lambda)\R\trans-\normal^\ast\otimes\R\av_1].
	\end{split}
\end{equation}
Since both $\R\av_1$ and $\R\av_2$ are vectors orthogonal to $\normal^\ast$, so also are the axial vectors associated with the skew-symmetric parts of both $\normal^\ast\otimes\R\av_1$ and $\normal^\ast\otimes\R\av_2$. Thus, combining \eqref{eq:conformal_nablast_c_ast} and \eqref{eq:conformal_intermediate}, we find that
\begin{equation}
	\label{eq:conformal_pre_curvature}
	\curls^\ast\cv^\ast\cdot\normal^\ast=\frac{1}{\lambda}\left(\curls\h\cdot\normal+\normal\times\nablas\ln\lambda\cdot\h\right).
\end{equation}
To conclude the proof, we must compute\footnote{Here use is made of the identity $$\curls(\av\times\bv)=(\nablas\av)\bv-(\nablas\bv)\av+(\divs\bv)\av-(\divs\av)\bv,$$ valid for generic vector fields $\av$ and $\bv$ on $\surface$.}
\begin{subequations}\label{eq:conformal_computations}
\begin{equation}
	\label{eq:conformal_computations_1}
	\curls\h\cdot\normal=\frac{1}{\lambda}\curls\cv\cdot\normal+\nablas\frac{1}{\lambda}\times\cv\cdot\normal+\nablas\frac{1}{\lambda}\cdot\nablas\ln\lambda+\frac{1}{\lambda}\laps\ln\lambda
\end{equation}
and 
\begin{equation}
	\label{eq:conformal_computations_2}
	\normal\times\nablas\ln\lambda\cdot\h=\frac{1}{\lambda}\left(\nablas\ln\lambda\times\cv\cdot\normal+|\nablas\ln\lambda|^2\right).
\end{equation}
\end{subequations}
The desired result follows from inserting \eqref{eq:conformal_computations} in \eqref{eq:conformal_pre_curvature}, as
\begin{equation*}
\nablas\frac{1}{\lambda}\cdot\nablas\ln\lambda=-\frac{1}{\lambda}|\nablas\ln\lambda|^2=-\frac{1}{\lambda^3}|\nablas\lambda|^2
\end{equation*}
and $K=-\curls\cv\cdot\normal$.
\end{proof}
\begin{remark}
	\label{rmk:conformal_interpretation}
	Equation \eqref{eq:conformal_gaussian_curvature} clearly shows how  the Gaussian curvature of a conformally deformed surface is uniquely determined by the Gaussian curvature of the undeformed surface and the metric tensor $\C$ in \eqref{eq:conformal_characterization}, thus confirming the intrinsic nature of $K$. Contrasting \eqref{eq:conformal_mean_curvature} to \eqref{eq:conformal_gaussian_curvature} shows instead how $H$ fails to be intrinsic, as in a conformal deformation it is also affected by the invariant rotation gradient $\H$.
\end{remark}
\begin{remark}
	\label{rmk:phi_definition}
	Setting $\lambda=\ex^\phi$ in \eqref{eq:conformal_gaussian_curvature}, we rewrite the latter in the equivalent form,
	\begin{equation}
		\label{eq:conformal_gaussian_curvature_rewritten}
		\laps\phi+K^\ast\ex^{2\phi}-K=0,
	\end{equation}
	which is more amenable to serve analytic purposes.
\end{remark}
\begin{remark}
	\label{rmk:moser}
	If $\surface$ is given (and so is $K$) and $K^\ast(\y(\x))$ is prescribed on $\surface$ as the Gaussian curvature of a target (unknown) surface $\surface^\ast$ related to $\surface$ by a conformal deformation $\y$, yet to be determined, equation  \eqref{eq:conformal_gaussian_curvature_rewritten} becomes a necessary condition for the existence of such a deformation.
	For the unit sphere $\sphere$ in three-dimensional space, \eqref{eq:conformal_gaussian_curvature_rewritten} reads as
	\begin{equation}
		\label{eq:conformal_gaussian_curvature_sphere}
		\laps\phi+K^\ast\ex^{2\phi}-1=0,
	\end{equation}
	which is an equation with an interesting history. Based on geometric intuition, one may be led to conjecture that \eqref{eq:conformal_gaussian_curvature_sphere} has always a solution provided that $K^\ast$ only takes on positive values. This is however false, as shown in \cite{kazdan:integrability}, where positive Gaussian curvatures $K^\ast$ unrealizable by conformal deformation of a sphere were explicitly constructed. A partial positive answer to the existence of conformally realizable curvatures was given in \cite{moser:nonlinear}, where for $K^\ast$ enjoying the \emph{antipoldal} symmetry, that is, such that $K^\ast(\y(-\x))=K^\ast(\y(\x))$, it was proved that \eqref{eq:conformal_gaussian_curvature_sphere} admits a solution with the same symmetry, provided that $\max_{\x\in\sphere} K^\ast(\y(\x))>0$.\footnote{A similar result had also been proved in \cite{koutroufiotis:gaussian}, but for $K^\ast$ close to a constant, so that $\surface^\ast$ is indeed an \emph{ovaloid}. The search for such a surface conformally obtained from a sphere was apparently a problem initially posed by Niremberg. The result proved in \cite{moser:nonlinear} removes the smallness condition enforced in \cite{koutroufiotis:gaussian}, thus providing a general solution to Niremberg's problem.}
\end{remark}
In the following section, we shall consider an even smaller class of deformations $\y$, for which the metric tensor $\C$ is simpler than in \eqref{eq:conformal_characterization} and, as we shall see, the metric implications are more stringent.

\section{Isometric deformations}\label{sec:isometric}
To restrict the range of conformal deformations of a surface, we first consider the class of all deformations that preserve area.
\begin{definition}
	\label{def:isoareal}
	A deformation $\y$ of $\surface$ is said to be \emph{isoareal} if
	\begin{equation}
	\label{eq:isoareal}
	A(\y(\mathscr{A}))= A(\mathscr{A})\quad\forall\ \mathscr{A}\subset\surface,
	\end{equation}
	where $A$ is the area measure.
\end{definition}
\begin{remark}
	\label{rmk:area_ratio}
	Since the element area ratio for a deformation $\y$ of $\surface$ into $\surface^\ast$ can be written as
	\begin{equation}
		\label{eq:area_ratio}
		\frac{\dd A^\ast}{\dd A}=\frac{|(\nablas\y)\uv\times(\nablas\y)\vv|}{|\uv\times\vv|}\quad\forall\ \uv,\vv\in\tplane_{\x} \text{ with } \uv\times\vv\neq\zero,
	\end{equation}
	it readily follows from \eqref{eq:deformation_gradient} that \eqref{eq:isoareal} can be equivalently written in the form (see also \cite{seguin:coordinate})
	\begin{equation}
		\label{eq:area_integral}
		\int_{\mathscr{A}}(\det\U-1)\dd A=0,
	\end{equation}
	and so \eqref{eq:isoareal} reduces to the local requirement
	\begin{equation}
		\label{eq:isoareal_local}
		\det\U=\det\V=1.
	\end{equation}
\end{remark}
\begin{definition}
	\label{def:isometry}
	A deformation $\y$ that preserves both angles and area is said to be an \emph{isometry}.\footnote{Here, \emph{isometry} and \emph{isometric deformation} are used as synonyms, although in differential geometry the former usually refers to a larger class of mappings than the latter (see Remark~\ref{rmk:reflection} below for a simple example of this difference.)}
\end{definition}
\begin{remark}
	\label{rmk:isometry}
	It follows from Proposition~\ref{prop:conformal_C} and \eqref{eq:isoareal_local} that an isometry of $\surface$ is characterized by having 
	\begin{subequations}
		\label{eq:isometry_deformation_gradient}
		\begin{equation}
			\label{eq:isometry_deformation_gradient_a}
			\nablas\y=\R\proj\quad\text{with}\quad\R\in\orth,
		\end{equation}
		for which
		\begin{equation}
			\label{eq:isometry_deformation_gradient_b}
			\C=\U=\proj\quad\text{and}\quad\B=\V=\mathbf{P}(\normal^\ast),
		\end{equation}
	\end{subequations}
	amounting to letting $\lambda\equiv1$ in \eqref{eq:conformal_characterization}.
\end{remark}
We now prove a number of properties enjoyed by isometries.
\begin{proposition}
	\label{prop:isometry_a_3}
	For an isometric deformation $\y$, the following equations are valid in general,
	\begin{equation}
		\label{eq:isometry_a_3}
		\av_3=\zero\quad\text{and}\quad K^\ast=K.
	\end{equation}
\end{proposition}
\begin{proof}
Since an isometry is a special conformal deformation, the two claims in \eqref{eq:isometry_a_3} are obtained by setting $\lambda\equiv1$ in \eqref{eq:conformal_a_3} and \eqref{eq:conformal_gaussian_curvature}, respectively.
\end{proof}
\begin{remark}
	\label{rmk:isometry_a_3}
	It follows from Proposition~\ref{prop:isometry_a_3} and \eqref{eq:wd_wb_a} that $\wdr=0$, and so no drilling elastic energy can be associated with an isometric deformation.
\end{remark}
\begin{remark}
	\label{rmk:theorema_egregium}
	The second claim in \eqref{eq:isometry_a_3} states that an isometry leaves the Gaussian curvature unchanged, which is one form of the classical \emph{theorema egregium} of Gauss (see, for example, \cite[p.\,140]{needham:visual}).
\end{remark}

\subsection{Pure rotation}\label{sec:rotations}
A rigid rotation of $\surface$ in three-dimensional space represented as $\y(\x)=\R\x$, where $\R$ is any given (constant) element of $\orth$, is clearly an isometry of $\surface$, as it satisfies \eqref{eq:isometry_deformation_gradient_a}. It is well-known that isometric deformations are far more general than uniform rotations. However, there are instances where an isometry is necessarily a uniform rotation. Two such such cases are examined below. 

If both $\surface$ and $\surface^\ast$ are the same sphere, say $\sphere$, for simplicity, then the only isometric deformations of $\surface$ onto $\surface^\ast$ are rigid motions in the ambient three-dimensional space. This is more formally stated as follows.
\begin{proposition}
	\label{prop:sphere}
	If $\y$ is an isometric deformation of $\sphere$ onto $\sphere$, then its polar rotation $\R$ is uniform, that is, $\nablas\R=\zero$.
\end{proposition}
\begin{proof}
	Since $\R$ is invertible, by \eqref{eq:H_definition}, the claim of the proposition is equivalent to say that $\H$ must vanish.
	
	A deformation $\y:\sphere\to\sphere$ can be represented as
	\begin{equation}
		\label{eq:sphere_deformation}
		\y(\x)=\R(\x)\x,
	\end{equation}
	where $\R(\x)\in\orth$. By differentiating both sides of \eqref{eq:sphere_deformation}, we arrive at 
	\begin{equation}
		\label{eq:sphere_deformation_gradient}
		\nablas\y=\R(\proj+\H\circ\normal),
	\end{equation}
	where we have used the identity $\normal(\x)=\x$ that holds on $\sphere$. By \eqref{eq:sphere_deformation_gradient} and \eqref{eq:H_representation}, $\y$ is an isometry if and only if
	\begin{equation}
		\label{eq:sphere_H_nu}
		\H\circ\normal=-\uv_2\otimes\av_1+\uv_1\otimes\av_2=\zero.
	\end{equation}
	Since $\uv_1$ and $\uv_2$ are linearly independent, \eqref{eq:sphere_H_nu} is equivalent to requiring that $\av_1=\av_2=\zero$, that is, again by \eqref{eq:H_representation}, that
	\begin{equation}
		\label{eq:sphere_H}
		\H=\W(\normal)\otimes\av_3,
	\end{equation}
	and so $\H$ vanishes by \eqref{eq:isometry_a_3}.
\end{proof}
\begin{remark}
	\label{rmk:reflection}
	Taking for $\R$ a reflection\footnote{So that $-\R$ is in $\orth$ instead of $\R$.} across a great circle of $\sphere$ would provide via \eqref{eq:sphere_deformation} an isometric map $\y$ of $\sphere$ onto itself, but such a  map would reverse the orientation of the surface, which (according to our definition) a deformation is not allowed to do.
\end{remark}
\begin{remark}
	\label{rmk:sphere_locality}
	As is clear from the proof of Proposition~\ref{prop:sphere}, the uniformity of $\R$ is a local property, which also applies if the deformation $\y$  just maps isometrically a part of $\sphere$ into another. Thus, this is a weaker form of the general result applied in \cite{sonnet:neutral} to identity a special class of M\"obius transformations as uniform rotations of the whole Riemann sphere.
\end{remark}
If an isometric deformation of a generic surface $\surface$ is further restricted, it may be reduced to a uniform rotation. One such restriction is the request of frame-indifference for the curvature tensor.
\begin{proposition}
	\label{prop:curvature_frame_indifference}
	Let $\y:\surface\to\euclid$ be an isometric deformation such that 
	\begin{equation}\label{eq:curvature_frame_indifference}
		\nablast\normal^\ast=\R\curvature\R\trans,
	\end{equation}
	where $\R$ is the polar rotation of $\y$. Then $\nablas\R=\zero$.\footnote{A different, perhaps longer proof of this result can be found in \cite[p.\,316]{o'neill:elementary}, see also \cite{grove:closed} for a similar result about convex surfaces.}
\end{proposition}
\begin{proof}
	To prove our claim, it suffices to show that $\H\equiv\zero$, that is, that $\av_1=\av_2=\zero$, since $\av_3=\zero$ for all isometries by Proposition~\ref{prop:isometry_a_3}. Now, combining \eqref{eq:connectors} and \eqref{eq:curvature_frame_indifference}, we obtain from \eqref{eq:R_representation} that
	\begin{equation}
		\label{eq:curvature_connectors}
		\dv^\ast_i=\R\dv_i\quad\text{for}\quad i=1,2.
	\end{equation}
	The desired conclusion then follows from setting $\lambda=1$ in \eqref{eq:conformal_a_2} and \eqref{eq:conformal_a_1}.
\end{proof}
\begin{remark}
	\label{rmk:curvature_connectors}
	While \eqref{eq:curvature_connectors} requires \eqref{eq:curvature_frame_indifference} to hold, its analog for the spin connector $\cv$,
	\begin{equation}
		\label{eq:c_connector}
		\cv^\ast=\R\cv,
	\end{equation}
	which follows from setting $\lambda=1$ in \eqref{eq:conformal_h},  is valid for all isometries.
\end{remark}

\subsection{Pure drilling}\label{sec:drilling}
We have already observed (see Remark~\ref{rmk:isometry_a_3}) that an isometric deformation of a surface comes with no drilling energy. It is then interesting to classify all  non-trivial isometries (different from a uniform rotation) that only entail drilling.\footnote{That is, for which there is a frame where $\Rb=\I$.} These would  carry no elastic energy of any other sort: no stretching energy, because $\wst$ in \eqref{eq:w_s} vanishes on isometries, and no bending energy either because $\wbe$ vanishes on a pure drilling mode (by construction). Were there a \emph{continuum} of energy-free deformations generating non-congruent shapes all in the same elastic ground state, they would  represent instances of \emph{soft elasticity} within the theory of shells proposed in \cite{sonnet:variational}. Such instances have already been proved to exist for minimal surfaces in \cite{sonnet:neutral}; here we show that these are indeed the only possible ones.
\begin{proposition}
	\label{prop:minimal_surfaces}
	An isometric deformation $\y$ of $\surface$ can be a pure drilling only if $\surface$ is a minimal surface, for which $\y$ reduces to a classical Bonnet transformation.\footnote{In the theory of minimal surfaces, the Bonnet transformation is responsible, in particular, for changing a catenoid into a helicoid (see, for example, pp.\,102 and 149 of \cite{dierkes:minimal}.)}
\end{proposition}
\begin{proof}
	We choose on $\surface$ the moving frame $\frameen$, where $(\e_1,\e_2)$ are the principal directions of curvature, so that \eqref{eq:gliding_laws} can be written with the curvature connectors $\dv_1$ and $\dv_2$ as in \eqref{eq:d_connectors_diagonal},
	\begin{equation}
		\label{eq:gliding_laws_curvature}
		\begin{cases}
			\nablas\e_1=\e_2\otimes\cv-\kappa_1\normal\otimes\e_1,\\
			\nablas\e_2=-\e_1\otimes\cv-\kappa_2\normal\otimes\e_2,\\
			\nablas\normal=\kappa_1\e_1\otimes\e_1+\kappa_2\e_2\otimes\e_2.
		\end{cases}
	\end{equation}
	
	An isometric deformation $\y$ of $\surface$ must obey \eqref{eq:isometry_deformation_gradient_a}, where $\R\in\mathsf{SO}(\normal)$ is a rotation about $\normal$, represented as
	\begin{equation}
		\label{eq:R_drilling_representation}
		\R=\cos\alpha(\e_1\otimes\e_1+\e_2\otimes\e_2)+\sin\alpha(\e_2\otimes\e_1-\e_1\otimes\e_2)+\normal\otimes\normal,
	\end{equation}
	with $\alpha$ a field to be determined to grant integrability to \eqref{eq:isometry_deformation_gradient_a}.\footnote{As made clear by \eqref{eq:R_drilling_representation}, in general $(\R\e_1,\R\e_2)$ are \emph{not} principal directions of curvature of $\surface$, although $(\R\e_1,\R\e_2,\normal)$ is a legitimate moving frame.} A tedious, but simple calculation elaborating on \eqref{eq:R_drilling_representation} and \eqref{eq:gliding_laws_curvature} gives the following formula for the invariant rotation gradient associated with $\R$,
	\begin{equation}
		\label{eq:H_drilling}
		\H=\W(\e_1)\otimes\av_1+\W(\e_2)\otimes\av_2+\W(\normal)\otimes\av_3,
	\end{equation}
	where
	\begin{equation}
		\label{eq:H_drilling_a}
		\av_1=\kappa_1\sin\alpha\e_1+\kappa_2(1-\cos\alpha)\e_2,\ \av_2=\kappa_1(\cos\alpha-1)\e_1+\kappa_2\sin\alpha\e_2,\ \av_3=\nablas\alpha.
	\end{equation}
	Thus, it follows from \eqref{eq:isometry_a_3}, \eqref{eq:conformal_a_1_a_2}, and \eqref{eq:mean_curvature} that $\alpha$ must be constant on $\surface$ and
	\begin{equation}
		\label{eq:drilling_integrability}
		\sin\alpha(\kappa_1+\kappa_2)=2H\sin\alpha =0,
	\end{equation}
	so that, unless $\alpha=0$ (in which case $\y$ reduces to the identity), $\surface$ must necessarily be a minimal surface.
	
	The proof is completed by the analysis of isometries of minimal surfaces conducted in \cite{sonnet:neutral} (see also \cite[p.\,159]{lawson:some}).
\end{proof}
\begin{remark}
	\label{rmk:H_star}
	By use of the first two equations in \eqref{eq:H_drilling_a} in \eqref{eq:conformal_mean_curvature} with $\lambda=1$, we readily obtain that
	\begin{equation}
		\label{eq:H_star}
		H^\ast=H\cos\alpha=0,
	\end{equation}
	in accord with the fact that the Bonnet transformation changes a minimal surface into another.
\end{remark}
\begin{remark}
	\label{rmk:w_b}
	It only takes a simple computation based on \eqref{eq:H_drilling_a} and \eqref{eq:d_connectors_diagonal} to check that $\wbe$ in \eqref{eq:wd_wb_a} vanishes identically for a pure drilling isometry. Thus, a whole family of non-trivial isometric deformations of $\surface$ is generated by varying the constant $\alpha$. For them all energy modes in \eqref{eq:w_s} and \eqref{eq:w_d_w_b} vanish; they indeed constitute an example of soft elasticity.
\end{remark}
\begin{remark}
	\label{rmk:conformal_extension}
	It follows from equations \eqref{eq:conformal_a}, \eqref{eq:conformal_a_1_a_2}, and \eqref{eq:conformal_mean_curvature} that both \eqref{eq:drilling_integrability} and the conclusions reached in Remarks~\ref{rmk:H_star} and \ref{rmk:w_b} remain valid for  \emph{uniformly} conformal deformations (that is, with constant $\lambda$, albeit different from unity). However, they are not a source of soft elasticity because $\wst\neq0$.
\end{remark}
\begin{remark}
	\label{rmk:bonnet}
	Pure drilling isometries are special non-trivial isometries that preserve the mean curvature of $\surface$. The surfaces for which such isometries exist are called the \emph{Bonnet surfaces}, as Bonnet first studied them (see \cite[pp.\,72-92]{bonnet:memoire} ).\footnote{Since all isometries preserve the Gaussian curvature, the Bonnet surfaces can equivalently be defined as the surfaces that admit isometries preserving both principal curvatures.} He proved that analytic surfaces with no umbilics\footnote{These are points with equal principal curvatures.} and \emph{constant} mean curvature are Bonnet surfaces.
	Further studies \cite{cartan:couples,chern:deformation,chern:moving} have been devoted to the characterization of these surfaces using the method of moving frames. In particular, new classes of Bonnet surfaces have been identified under the assumption that they contain no umbilics and are at least of class $C^5$ \cite{chern:deformation}. Moreover, it was proved in \cite{bracken:helicoidal} that \emph{helicoidal} surfaces\footnote{One such surface is produced by a helicoidal motion in three-dimensional space of an appropriate curve.} are Bonnet surfaces.
\end{remark}

\subsection{Pure bending}\label{sec:bending}
Another, complementary way to restrict an isometric deformation of a material surface $\surface$ is by requiring it to be a pure bending (in some frame). This is achieved, in particular, by replacing $\R$ in \eqref{eq:R_drilling_representation} with
\begin{equation}
	\label{eq:R_bending_representation}
	\R=\e_1\otimes\e_1+\cos\alpha(\e_2\otimes\e_2+\normal\otimes\normal)+\sin\alpha(\normal\otimes\e_2-\e_2\otimes\normal),
\end{equation}
which describes a rotation by angle $\alpha$ about $\e_1$.

We shall identify pure bending isometries in a special class of deformations, which are defined as follows.
\begin{definition}
	\label{def:eversion}
	We say that an isometric deformation $\y$ of $\surface$ is an \emph{eversion} if
	\begin{equation}
		\label{eq:eversion_definition}
		\nablast\normal^\ast(\y(x))=-\R(\x)\nablas\normal(\x)\R(\x)\trans\quad\forall\ \x\in\surface,
	\end{equation}
	where $\R$ is the polar rotation of $\y$.
\end{definition}
\begin{proposition}\label{prop:eversion}
	A pure bending isometry  of $\surface$  with polar rotation as in \eqref{eq:R_bending_representation}, where $\e_1$ is a  principal direction of curvature, is an \emph{eversion} if an only if the following conditions hold,
	\begin{subequations}
		\label{eq:eversion_conditions}
		\begin{align}
			\nablas\kappa_1\cdot\e_2&=\kappa_1(\kappa_2-\kappa_1)\frac{\sin\alpha}{1-\cos\alpha},\label{eq:eversion_condition_a}\\
			\nablas\kappa_2\cdot\e_1&=0,\label{eq:eversion_condition_b}\\
		\nablas\alpha&=2\kappa_2\e_2.\label{eq:eversion_condition_c},
		\end{align}
	\end{subequations}
	where $\kappa_1$ and $\kappa_2$ are the principal curvatures of $\surface$.
\end{proposition}
\begin{proof}
A direct computation resorting to \eqref{eq:gliding_laws_curvature} delivers $\H$ in the form \eqref{eq:H_drilling}, where now
\begin{subequations}
	\label{eq:bending_a}
	\begin{align}
		\av_1&=\nablas\alpha,\label{eq:bending_a_1}\\
		\av_2&=\sin\alpha\cv+\kappa_1(\cos\alpha-1)\e_1,\label{eq:bending_a_2}\\
		\av_3&=(\cos\alpha-1)\cv-\kappa_1\sin\alpha\e_1.\label{eq:bending_a_3}
	\end{align}
\end{subequations}
Since $\av_3$ must vanish for an isometry, we derive from \eqref{eq:bending_a_3} that
\begin{equation}
	\label{eq:bending_c}
	\cv=-\frac{\kappa_1\sin\alpha}{1-\cos\alpha}\e_1.
\end{equation}
It then follows from \eqref{eq:bending_c} and \eqref{eq:bending_a_2} that
\begin{equation}
	\label{eq:bending_a_2_explicit}
	\av_2=-2\kappa_1\e_1,
\end{equation}
which combined with \eqref{eq:bending_a_1} shows that the integrability condition \eqref{eq:conformal_a_1_a_2} reduces to
\begin{equation}
	\label{eq:bending_integrability}
	\nablas\alpha\cdot\e_1=0.
\end{equation}
Equations \eqref{eq:eversion_condition_a} and \eqref{eq:eversion_condition_b} follow from inserting \eqref{eq:bending_c} into equations \eqref{eq:mainardi_special}.

We now use equations \eqref{eq:conformal_a_2} and \eqref{eq:conformal_a_1} combined with \eqref{eq:bending_a_1} and \eqref{eq:bending_a_2_explicit} to obtain
\begin{equation}
	\label{eq:eversion_d_ast}
	\dv_1^\ast=\kappa_1\R\e_1=\kappa_1\e_1\quad\text{and}\quad \dv_2^\ast=((\nablas\alpha\cdot\e_2)-\kappa_2)\R\e_2,
\end{equation}
which lead us to the following expressions for the principal curvatures of $\surface^\ast$,
\begin{equation}
	\label{eq:eversion_kappa_ast}
	\kappa_1^\ast=-\kappa_1\quad\text{and}\quad\kappa_2^\ast=\kappa_2-\nablas\alpha\cdot\e_2.
\end{equation}
Thus, \eqref{eq:eversion_definition} is obeyed if and only if \eqref{eq:eversion_condition_c} is.
\end{proof}
At first sight conditions \eqref{eq:eversion_conditions} may appear to be too many to be compatible; we shall now see that this is not the case.
\begin{proposition}
	\label{prop:revolution}
	All surfaces of revolution admit a pure bending eversion.
\end{proposition}
\begin{proof}
	In the frame $\framecyl$ of ordinary cylindrical coordinates $(r,\vt,z)$, we represent a surface of revolution about $\e_z$ as
	\begin{equation}
		\label{eq:revolution_surface}
		\x(\vt,z)=\rho(z)\e_r+z\e_z,\quad\vt\in[0,2\pi),\ z\geqq0,
	\end{equation}
	where $\rho$ is a smooth positive function, with $\rho(z)$ representing the radius of the circular section of $\surface$ at  height $z$. The parallels of $\surface$ are oriented along $\e_1=\e_\vt$ and its meridians along
	\begin{equation*}
		\label{eq:revolution_e_2}
		\e_2=\frac{\rho'\e_r+\e_z}{\sqrt{1+\rho'^2}},
	\end{equation*}
	while the unit normal $\normal$ is given by
	\begin{equation}
		\label{eq:revolution_nu}
		\nu=\e_1\times\e_2=\frac{\e_r-\rho'\e_z}{\sqrt{1+\rho'^2}}.
	\end{equation}
The principal curvatures associated with the principal directions $(\e_1,\e_2)$ are
\begin{equation}
	\label{eq:revolution_kappa}
	\kappa_1=\frac{1}{\rho\sqrt{1+\rho'^2}}\quad\text{and}\quad\kappa_2=-\frac{\rho''}{(1+\rho'^2)^{3/2}}.
\end{equation}

Equation \eqref{eq:eversion_condition_b} is thus automatically satisfied. To satisfy \eqref{eq:eversion_condition_c}, we take $\alpha=\alpha(z)$ and compute
\begin{equation}\label{eq:revolution_nabla_alpha}
	\nablas\alpha=\frac{\alpha'}{\sqrt{1+\rho'^2}}\e_2.
\end{equation}
Similarly, we obtain $\nablas\kappa_2\cdot\e_2$, so that equations \eqref{eq:eversion_condition_a} and \eqref{eq:eversion_condition_c} become
\begin{equation}
	\label{eq:revolution_system}
	\begin{cases}
		\displaystyle\rho'=\frac{\sin\alpha}{1-\cos\alpha},\\\\
	    \displaystyle\alpha'=-\frac{2\rho''}{1+\rho'^2}.
	\end{cases}	
\end{equation}
It is straightforward to check that for any function $\rho$ this system is solved by\footnote{The two-argument function $\arctan(y,x)$, ranging in the interval $[-\pi,\pi]$, extends the standard function $\arctan(y/x)$ ranging in $[-\pi/2,\pi/2]$; it attributes to the correct quadrant of the Cartesian plane the angle subtended by the positive $x$-axis and the radial line through the origin and the point $(x,y)$.}
\begin{equation}
	\label{eq:revolution_alpha_solution}
	\alpha=\arctan(2\rho',\rho'^2-1).
\end{equation}
In Appendix~\ref{sec:eversion}, we exhibit the deformation $\y$ corresponding to \eqref{eq:revolution_alpha_solution} together with other computational details.
\end{proof}
Figure~\ref{fig:eversion} shows the eversion of half a catenoid represented by \eqref{eq:revolution_surface} with $\rho(z)=\cosh z$ for $0\leqq z\leqq2$.
\begin{figure}[] 
	\centering
	\begin{subfigure}[b]{0.4\linewidth}
		\centering
		\includegraphics[width=\linewidth]{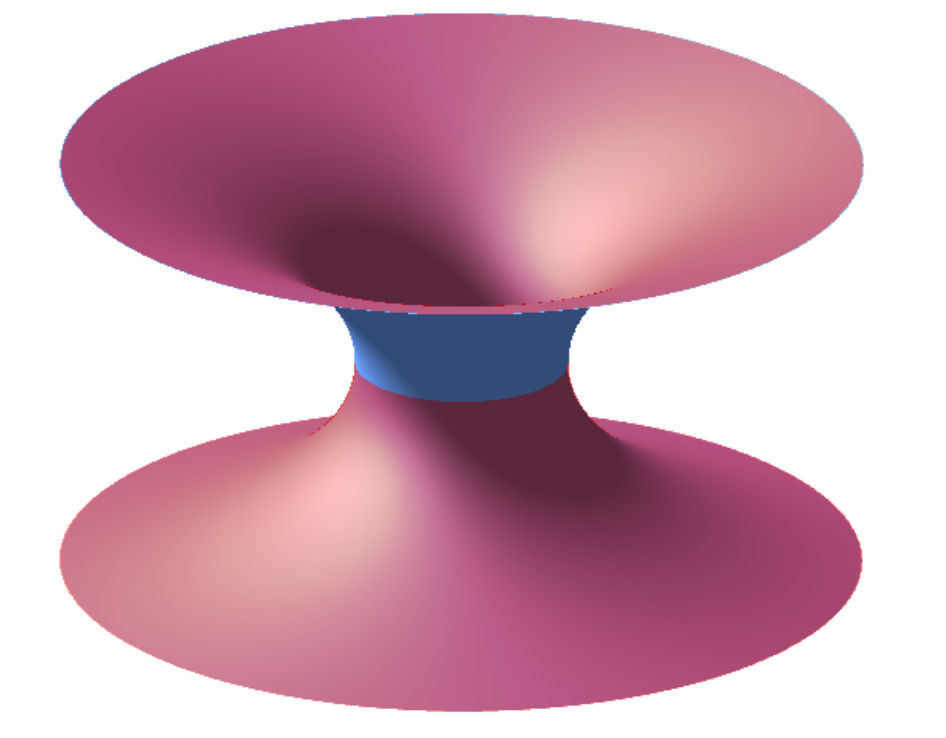}
		\caption{Top view.}
		\label{fig:eversion_a}
	\end{subfigure}
	\qquad
	\begin{subfigure}[b]{0.4\linewidth}
		\centering
		\includegraphics[width=\linewidth]{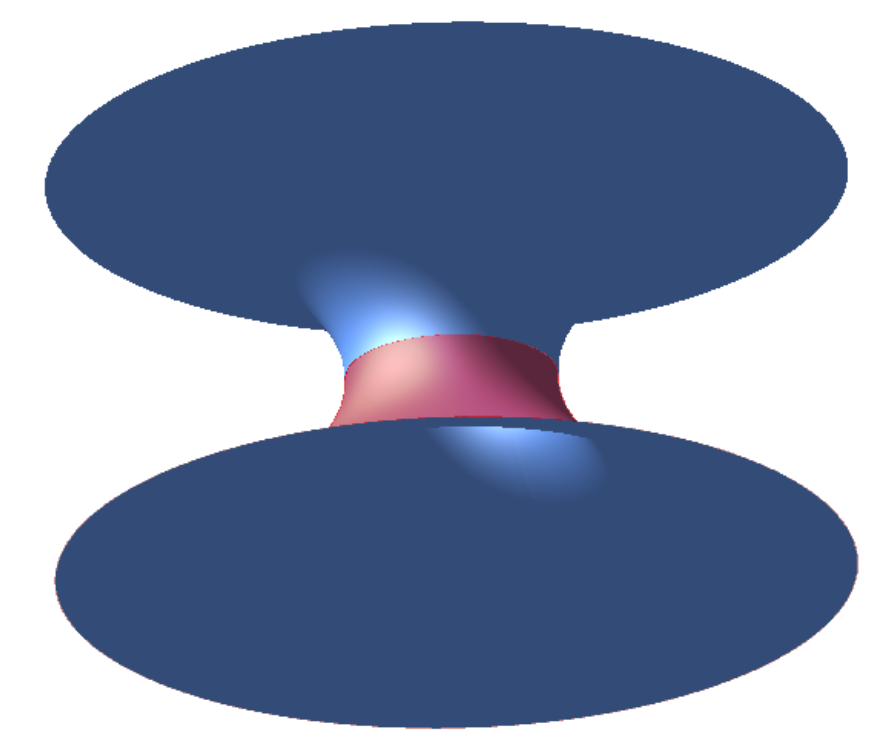}
		\caption{Bottom view.}
		\label{fig:eversion_b}
	\end{subfigure}
	\caption{Two views of the eversion of half a catenoid based on a unit circle and extending for $0\leqq z\leqq2$. The surface is everted inside out as suggested by the different colours of the exposed side.}
	\label{fig:eversion}
\end{figure}
\begin{remark}
	\label{rmk:contrast}
	It is instructive to contrast Proposition~\ref{prop:curvature_frame_indifference} with Proposition~\ref{prop:eversion}: what is a pure rotation in the former becomes an eversion in the latter, while only a sign makes \eqref{eq:curvature_frame_indifference} differ from \eqref{eq:eversion_definition}.
\end{remark}
\begin{remark}
	\label{rmk:eversion_w_b}
	Since for an isometry $\V=\mathbf{P}(\normal^\ast)$, it follows from \eqref{eq:wb_connectors} and \eqref{eq:eversion_definition} that $\wbe=0$ for a pure bending eversion. However, despite the fact that both $\wst$ and $\wdr$ also vanish, an eversion is \emph{not} an instance of soft elasticity (see Remark~\ref{rmk:w_b}), as the everted shape cannot be attained through a family of energy-free deformations; quite on the contrary, it is \emph{energetically isolated}: a surface must in general go through a great deal of stretching (and possibly drilling too) to get everted.
\end{remark}
\begin{remark}
	\label{rmk:truesdell}
	Elastic eversion was perhaps first studied (both theoretically and ``experimentally'') by Truesdell (see, in particular pp.\,510-519 of \cite{truesdell:some}). In more recent times, the eversion of shells has again become quite a popular topic (see, for example, \cite{krieger:extreme,reis:perspective,taffetani:static,liu:snap,taffetani:nonlinear}) also involving some exotic applications that exploit the bistability of everted shapes \cite{kwak:edible}.
\end{remark}

\section{Conclusions}\label{sec:conclusions}
In deforming a material surface in three-dimensional space, angles and areas will generally be altered. When they are not, the deformation is an isometry and this metric restriction, which is geometrically intrinsic in nature, has notable implications, the most  known of which is perhaps Gauss' \emph{theorema egregium}, which states that the Gaussian curvature of the surface remains unchanged under the deformation. We have attempted to explore other, possibly more exotic implications of metric restrictions imposed on a deformation of a material surface.

As our mathematical language betrays, we have been interested in kinematics, seen as the geometric vestments of mechanics: in this view, a material surface cannot be identified with its fundamental forms; it is rather a coherent collection of individual body-points that move in space and cannot interpenetrate. Such a conception of material surfaces poses a number of restrictions to its geometric manifestations. We have \emph{not} used the classical method of coordinates, although these are unavoidable to accomplish certain computational tasks.\footnote{In this respect, our style is reminiscent of that first fashioned in \cite{gurtin:continuum,gurtin:addenda}, of which \cite{seguin:coordinate} represents perhaps the most recent emanation.} We have embraced the method of moving frames, but in a vectorial variant, which avoids resorting to differential forms, and is thus possibly more germane to the taste and skills of the engineering community.

Although some preliminaries on surface calculus were needed to make the reader acquainted with our
approach, our endeavour was not only limited to reobtaining in a different way results that have already been known for quite some time. A method is best tested by new applications. We found these in a theory for soft thin shells that we have recently proposed. This is a theory based on the energetic separation of three independent deformation modes: stretching, drilling, and bending. We explored the consequences that some metric restrictions of the deformation have on the three independent energetic contents.

Here is a list of our main conclusions:
\begin{enumerate}[(1)]
	\item An isometric deformation, which has no stretching content, has no drilling content either.
	\item An isometric deformation of any part of a sphere onto another is a uniform rotation in the ambient space.
	\item An isometric deformation of a surface $\surface$ has no bending elastic energy (besides having no drilling elastic energy) only if $\surface$ is a minimal surface and the deformation is a Bonnet transformation.
	\item All surfaces of revolution admit a pure bending eversion with no elastic energy.
\end{enumerate} 

We have given more general conditions for the existence of pure bending eversions; they might be worth exploring further for classes of shell shapes also lacking any particular symmetry. It is our hope that such a study could be conducted in the future; it might shed some new light on the snapping of soft shells between bistable configurations.


\backmatter


\begin{appendices}
	
	\section{Metric curvature formula}\label{sec:metric}
	In this appendix, we derive from the Gauss equation \eqref{eq:gauss_equation} an elegant formula for the Gaussian curvature $K$, which is called \emph{metric} as it expresses $K$ solely in terms of the metric elements $|\rvu|$ and $|\rvv|$ defined in Sect.~\ref{sec:prelim} for a generic system of coordinates. Here, we shall assume that $(u,v)$ coordinates are merely orthogonal, and $|\rvu|$ and $|\rvv|$ may differ.
	
	We start by recalling how the surface gradient $\nablas f$ of a smooth, scalar-valued function $f$ defined on $\surface$ is related to the partial derivatives $\partial_uf$ and $\partial_vf$ of $f$ expressed in the $(u,v)$ coordinates. For a generic curve $t\mapsto\rv(t)$, parameterized as $t\mapsto(u(u),v(t))$,
	\begin{equation}
		\label{eq:f_dot}
		\dot{f}(\rv(t))=\nablas f\cdot\dot{\rv}=\nablas f\cdot(\dot{u}|\rvu|\e_u+\dot{v}|\rvv|\e_v),
	\end{equation}
	where a superimposed dot denotes differentiation with respect to $t$ and use has been made of \eqref{eq:e_u_e_v_definition}. Since, for  $f$ expressed in $(u,v)$ coordinates, we also have that $\dot{f}=(\partial_uf)\dot{u}+(\partial_vf)\dot{v}$, and the curve along which the differentiation is performed is arbitrary, we easily derive from \eqref{eq:f_dot} that 
	\begin{equation}
		\label{eq:nablas_f}
		\nablas f=\frac{\partial_uf}{|\rvu|}\e_u+\frac{\partial_vf}{|\rvv|}\e_v.
	\end{equation}
	
	Letting $(\cv,\dv_u,\dv_v)$ denote the connectors of the moving frame $\framee$ associated with orthogonal coordinates $(u,v)$, we wish now to represent $\cv$ in the basis $(\e_u,\e_v)$. To this end, since by \eqref{eq:gliding_laws} $\cv=(\nablas\e_u)\trans\e_v$, with the aid of \eqref{eq:nablas_f} we first compute 
	\begin{equation}\label{eq:nablas_e_u}
		\nablas\e_u=-\frac{1}{|\rvu|}\e_u\otimes\nablas|\rvu|+\frac{1}{|\rvu|}\left(\frac{1}{|\rvu|}\partial_{uu}\rvu\otimes\e_u+\frac{1}{|\rvv|}\partial_{vu}\rvu\otimes\e_v\right),
	\end{equation}
	from which we obtain
	\begin{equation}\label{eq:c}
		\cv=\frac{1}{|\rvu|^2|\rvv|}(\partial_{uu}\rvu\cdot\partial_v\rv)\e_u+\frac{1}{|\rvv|^2|\rvu|}(\partial_{vu}\rv\cdot\partial_v\rv)\e_v.	
	\end{equation}
	Since $\partial_u\rv\cdot\partial_v\rv=0$ and $\partial_{uv}\rv=\partial_{vu}\rv$, again by use of \eqref{eq:e_u_e_v_definition}, \eqref{eq:c} can be given a more symmetric form,
	\begin{equation}
		\label{eq:c_symmetric}
		\cv=\frac{1}{|\rvu||\rvv|}(\e_v\otimes\e_v-\e_u\otimes\e_u)\partial_{uv}\rv.
	\end{equation}
	
	To express $\cv$ only in terms of derivatives of the metric elements, we remark that 
	\begin{subequations}
		\label{eq:c_remark}
		\begin{align}
			\partial_{uv}\rv=\partial_u(|\rvv|\e_v)&=(\partial_u|\rvv|)\e_v+|\rvv|\partial_u\e_v\\
			&=(\partial_v|\rvu|)\e_u+|\rvu|\partial_v\e_u.
		\end{align}
	\end{subequations}
	Equations \eqref{eq:c_symmetric} and \eqref{eq:c_remark} together with the identities $\e_u\cdot\partial_v\e_u=\e_v\cdot\partial_u\e_v=0$, which follow from both $\e_u$ and $\e_v$ being unit vectors, allow us to decompose $\cv$ as $\cv=c_u\e_u+c_v\e_v$, where
	\begin{equation}
		\label{eq:c_u_c_v}
		c_u=-\frac{1}{|\rvu||\rvv|}\partial_v|\rvu|,\quad c_v=\frac{1}{|\rvu||\rvv|}\partial_u|\rvv|.
	\end{equation}
	
	With this representation for $\cv$ at hand, we can now make use of \eqref{eq:gauss_equation} to obtain the desired formula for $K$. From \eqref{eq:gliding_laws} written for the frame $\framee$ we easily arrive at 
	\begin{equation}
		\label{eq:curls_e_u_e_v}
		\curls\e_u\cdot\normal=c_u\quad\text{and}\quad\curls\e_v\cdot\normal=c_v,
	\end{equation}
	and so
	\begin{align}
		\label{eq:pre_K}
		\curls\cv\cdot\normal&=\curls(c_u\e_u+c_v\e_v)\cdot\normal\nonumber\\
		&=c_u^2+c_v^2+\e_u\cdot\nablas c_v-\e_v\cdot\nablas c_u.
	\end{align}
	By inserting \eqref{eq:c_u_c_v} in \eqref{eq:pre_K}, after some simplifications, we conclude that 
	\begin{equation}
		\label{eq:K_metric_formula}
		K=-\frac{1}{|\rvu||\rvv|}\left[\partial_u\left(\frac{\partial_u|\rvv|}{|\rvu|}\right)+\partial_v\left(\frac{\partial_v|\rvu|}{|\rvv|}\right)\right].
	\end{equation}
	This equation can be found in all elementary textbooks on differential geometry.\footnote{See, for example, \cite[p.\,297]{o'neill:elementary} or \cite[equation (38.35)]{needham:visual}, to cite just two whose development is closer in style to the one presented here.}
	
	We finally note that equation \eqref{eq:pre_K}, written here for the coordinate frame $\framee$, is indeed valid for a generic moving frame $\frameen$, 
	as it simply relies on the gliding laws \eqref{eq:gliding_laws} in the main text. Once applied to the moving frame identified by the principal directions of curvature, and combined with \eqref{eq:gauss_equation}, \eqref{eq:pre_K} delivers the following alternative curvature formula,
	\begin{equation}\label{eq:K_alternative}
K=\e_2\cdot\nablas(\cv\cdot\e_1)-\e_1\cdot\nablas(\cv\cdot\e_2)-(\cv\cdot\e_1)^2-(\cv\cdot\e_2)^2,		
\end{equation}
see also (38.36) of \cite{needham:visual}.

\section{Drilling and bending contents}\label{sec:contents}
As shown in \cite{sonnet:variational}, the drilling and bending \emph{contents} of $\y$ are vector fields on $\surface$ given by
\begin{equation}
	\label{eq:contents}
	\dv=a_\nu\normal,\quad\bv=\frac{1}{1+a_\nu^2}(\I+a_\nu\W(\normal)\proj)\av\quad\text{with}\quad a_\nu:=\av\cdot\normal,
\end{equation}
where $\I$ is the identity tensor (in three dimensions), $\proj:=\I-\normal\otimes\normal$ is the projection onto the plane orthogonal to $\normal$ (the tangent plane), $\W(\normal)$ is the skew-symmetric tensor associated with $\normal$,\footnote{For a generic vector $\uv$, the skew-symmetric tensor $\W(\uv)$ associated with $\uv$ is such that $\W(\uv)\vv=\uv\times\vv$, for all vectors $\vv$.} and $\av$ is the vector representing $\R$, which can be extracted from
\begin{equation}
\label{eq:W(a)}
\W(\av)=\frac{\R-\R\trans}{1+\tr\R}.
\end{equation}

Both $\dv$ and $\bv$ are frame-dependent, whereas, as established in \cite{sonnet:variational}, both $\wdr$ and $\wbe$ in \eqref{eq:w_d_w_b} are frame-invariant measures of drilling and bending, respectively.

\section{Integrability condition}\label{sec:integrability}
In this appendix, we complete the proof of equations \eqref{eq:integrability_conditions}. By \eqref{eq:deformation_gradient}, \eqref{eq:H_definition}, and \eqref{eq:H_representation}, we can write
\begin{align}
	\label{eq:second_deformation_gradient}
	\nablastwo\y&=\nablas(\R\U)=\R\W(\uv_1)\U\otimes\av_1+\R\W(\uv_2)\U\otimes\av_2+\R\W(\normal)\U\otimes\av_3\nonumber\\
	&+\R\nablas(\lambda_1\uv_1\otimes\uv_1+\lambda_2\uv_2\otimes\uv_2),
\end{align}
where \eqref{eq:U_V} has also been used. Equations \eqref{eq:gliding_laws} and \eqref{eq:V_v} then allow us to expand \eqref{eq:second_deformation_gradient} in the lengthy, but useful form,
\begin{align}
	\label{eq:second_deformation_gradient_expanded}
	\nablastwo\y&=\vv_1\otimes(-\lambda_2\uv_2\otimes\av_2+\uv_1\otimes\nablas\lambda_1+\lambda_1\uv_2\otimes\cv+\lambda_1\normal\otimes\dv_1-\lambda_2\uv_2\otimes\cv)\nonumber\\
	&+\vv_2\otimes(\lambda_1\uv_1\otimes\av_3+\uv_2\otimes\nablas\lambda_2+\lambda_1\uv_1\otimes\cv+\lambda_2\normal\otimes\dv_2-\lambda_2\uv_1\otimes\cv)\nonumber\\
	&+\normal^\ast\otimes(\lambda_2\uv_2\otimes\av_1-\lambda_1\uv_1\otimes\av_2+\lambda_1\uv_1\otimes\dv_1+\lambda_2\uv_2\otimes\dv_2).
\end{align}
On the other hand,
\begin{equation}
(\nablas\y)\curvature\otimes\normal=-\lambda_1\vv_1\otimes\dv_1\otimes\normal-\lambda_2\vv_2\otimes\dv_2\otimes\normal,
\end{equation}
so that, since $\framev$ is a basis of $\transl$, with little more labour, setting  $\F=\nablas\y$ in \eqref{eq:integrability_vector} delivers \eqref{eq:integrability_conditions}.
	
\section{Eversion of a surface of revolution}\label{sec:eversion}
On a surface of revolution $\surface$ represented as in \eqref{eq:revolution_surface} of the main text, we consider a smooth curve parameterized by $(\vt(t),z(t))$. Differentiating this curve in the parameter $t$,\footnote{Operation denoted by a superimposed dot.} we obtain from \eqref{eq:revolution_surface} that
\begin{equation}
	\label{eq:revolution_x_dot}
	\dot\x=\rho\dot{\vt}\e_1+\sqrt{1+\rho'^2}\dot{z}\e_2.
\end{equation}
Thus, for a scalar field $\alpha$ depending only on $z$,
\begin{equation}
	\label{eq:revolution_alpha_dot}
	\dot\alpha=\alpha'\dot{z}=\nablas\alpha\cdot\dot{\x},
\end{equation}
which together with \eqref{eq:revolution_x_dot} implies \eqref{eq:revolution_nabla_alpha}.

Let now $\y$ be the deformation that maps $\x(\vt,z)$ into
\begin{equation}
	\label{eq:revolution_y}
	\y(\x)=\rho(z)\e_r-z\e_z.
\end{equation}
Its effect on half a catenoid is shown in Fig.~\ref{fig:eversion}. Along the same curve introduced above,
\begin{equation}
	\label{eq:revolution_y_dot}
	\dot\y=\rho\dot{\vt}\e_1\vt+\dot{z}(\rho'\e_r-\e_z),
\end{equation}
which, as
\begin{equation}
	\label{eq:revolution_e_r_e_vt_e_z}
	\e_r=\frac{\rho'\e_2+\normal}{\sqrt{1+\rho'^2}},\quad\e_\vt=\e_1,\quad\e_z=\frac{\e_2-\rho'\normal}{\sqrt{1+\rho'^2}},
\end{equation}
can also be written as 
\begin{equation}
	\label{eq:revolution_y_dot_rewritten}
	\dot\y=\rho\dot{\vt}\e_1+\frac{\dot{z}}{\sqrt{1+\rho'^2}}[(\rho'^2-1)\e_2+2\rho'\normal].
\end{equation}
Comparing \eqref{eq:revolution_y_dot_rewritten} and \eqref{eq:revolution_x_dot}, we obtain that 
\begin{equation}
	\label{eq:revolution_deformation_gradient}
	\nablas\y=\frac{1}{1+\rho'^2}[(\rho'^2-1)\e_2+2\rho'\normal]+\e_1\otimes\e_1,
\end{equation}
which can be given the following form derived from \eqref{eq:R_bending_representation},
\begin{equation}
	\label{eq:revolution_deformation_gradient_equivalent}
	\R\proj=\e_1\otimes\e_1+\cos\alpha\e_2\otimes\e_2+\sin\alpha\normal\otimes\e_2,
\end{equation}
only if
\begin{equation}
	\label{eq:revolution_alpha}
	\sin\alpha=\frac{2\rho'}{1+\rho'^2}\quad\text{and}\quad\cos\alpha=\frac{\rho'^2-1}{1+\rho'^2},
\end{equation}
which is equivalent to \eqref{eq:revolution_alpha_solution}.
\end{appendices}


\end{document}